\documentclass[11pt]{article}

\newcommand{\dt}{September 16, 2008}

\usepackage{amsfonts}
\usepackage{url}
\usepackage{amsmath,amsthm}
\usepackage{graphicx}
\usepackage{amssymb}
\usepackage{rotating}
\usepackage{epsfig}
\newcommand{\bk}{\bullet}

\newcommand{\ds}{\mbox{\rm Dis}}

\newcommand{\ci}{\circ}
\newcommand{\D}{\bullet}

\newcommand{\G}{\mathcal{G}}
\newcommand{\tr}{{\rm trace}}
\newcommand{\qc}{{quasi-complete }}
\newcommand{\QC}{ {\rm QC} }
\newcommand{\qs}{{quasi-star  }}
\newcommand{\QS}{ {\rm QS} }
\newcommand{\thr}{ {\rm Th} }
\newcommand{\qcc}{ {\rm qc} }
\newcommand{\qss}{ {\rm qs} }
\newcommand{\adj}{ {\rm Adj} }
\newcommand{\DF}{ {\rm Diff} }
\newcommand{\sd}[1]{\begin{sideways}$#1$\end{sideways}}

\newtheorem{theorem}{Theorem}

\newtheorem{corollary}{Corollary}

\newtheorem{lemma}{Lemma}

\begin{document}
\begin{sloppypar}

\pagestyle{myheadings} \markboth{\rm \sc \footnotesize \'{A}brego,
Fern\'{a}ndez, Neubauer, Watkins, \dt, \today}{\rm \sc \footnotesize
\'{A}brego, Fern\'{a}ndez, Neubauer, Watkins, \dt, \today}

\title{Sum of squares of degrees in a graph\footnotetext{\hskip-0.58truecm
$^1$ Partially supported by CIMAT.\\
 Current address: Centro de Investigaci\'on en Matem\'aticas, A.C.,\\
 Jalisco S/N, Colonia Valenciana, 36240, Guanajuato, GTO., M\'exico.}}
\author{Bernardo M. {\'A}brego $^1$\\
Silvia Fern{\'a}ndez-Merchant$^{1}$\\
Michael G. Neubauer\\
William Watkins\\ \\
Department of Mathematics\\
California State University, Northridge\\
18111 Nordhoff St, Northridge, CA, 91330-8313, USA\\
\footnotesize{email:\texttt{\{bernardo.abrego, silvia.fernandez,}}\\ \footnotesize{\texttt{\hskip 2truecm michael.neubauer, bill.watkins\}@csun.edu}}
}

\date{\dt \\ \today }
%\footnote{
%\noindent Department of Mathematics, California State University, Northridge, CA, 91330-8313, USA, bill.watkins@csun.edu}
\maketitle

\begin{abstract}
Let $\G(v,e)$ be the set of all simple graphs with $v$ vertices
and $e$ edges and let  $P_2(G)=\sum d_i^2$ denote the sum of the squares of the degrees, $d_1, \ldots, d_v$, of the vertices of $G$.

It is known that the maximum value of $P_2(G)$ for $G \in \G(v,e)$ occurs at one or both of two special graphs in $\G(v,e)$---the \qs graph or the \qc graph.  For each pair $(v,e)$, we determine which of these two graphs has the larger value of $P_2(G)$.  We also determine all pairs $(v,e)$ for which the values of $P_2(G)$ are the same for the \qs and the \qc graph.  In addition to the \qs and \qc graphs, we  find all other graphs in $\G(v,e)$ for which the maximum value of $P_2(G)$ is attained.  Density questions posed by previous authors are examined.

 \end{abstract}


\newpage
\section{Introduction} \label{sec:intro}
Let $\G(v,e)$ be the set of all simple graphs with $v$ vertices
and $e$ edges and let  $P_2(G)=\sum d_i^2$ denote the sum of the squares of the degrees, $d_1, \ldots, d_v$, of the vertices of $G$.
The purpose of this paper is to finish the solution of an old
problem:
\begin{enumerate}
\item What is the maximum value of $P_2(G)$, for a graph $G$ in
$\G(v,e)$? \item For which graphs $G$ in $\G(v,e)$ is the maximum
value of $P_2(G)$ attained?
\end{enumerate}

Throughout, we  say that a graph $G$ is {\it optimal} in
$\G(v,e)$, if $P_2(G)$ is maximum and we denote this maximum value
by $\max(v,e)$.

These problems were first investigated by Katz \cite{Ka} in 1971 and
by R. Ahlswede and G.O.H. Katona \cite{AK} in 1978.  In his review
of the paper by  Ahlswede and Katona. P. Erd\H{o}s \cite{Er}
commented that ``the solution is more difficult than one would
expect." Ahlswede and Katona were interested in an equivalent form
of the problem: they wanted to find the maximum number of pairs of
different edges that have a common vertex.  In other words, they
want to maximize the number of edges in the line graph $L(G)$ as $G$
ranges over $\G(v,e)$.  That these two formulations of the problem
are equivalent follows from an examination of the vertex-edge
incidence matrix $N$ for a graph $G \in \G(v,e)$:
\begin{eqnarray*}
\tr((NN^T)^2)&=& P_2(G)+2e\\
\tr((N^TN)^2)&=&\tr(A_L(G)^2)+4e,
\end{eqnarray*}
where $A_L(G)$ is the adjacency matrix of the line graph of $G$.
Thus $P_2(G)=\tr(A_L(G)^2)+2e$. ($\tr(A_L(G)^2)$ is twice the number
of edges in the line graph of $G$.)

Ahlswede and Katona showed that the maximum value $\max(v,e)$ is
always attained at one or both of two special graphs in $\G(v,e)$.

They called the first of the two special graphs a {\it
quasi-complete} graph. The \qc graph in $\G(v,e)$ has the largest possible
complete subgraph $K_k$.  Let $k,j$ be the unique integers such that

\[
e=\binom{k+1}{2}-j = \binom{k}{2}+k-j,\mbox{ where } 1 \leq j \leq
k.
\]

The \qc graph in $\G(v,e)$, which is denoted by $\QC(v,e)$, is
obtained from the complete graph on the $k$ vertices $1,2, \ldots,
k$ by adding $v-k$ vertices $k+1, k+2, \ldots, v$, and the edges
$(1,k+1)$, $(2,k+1), \ldots, (k-j,k+1)$.

The other special graph in $\G(v,e)$ is the {\it quasi-star}, which
we denote by $\QS(v,e)$.  This graph has as many dominant vertices
as possible.  (A {\it dominant vertex} is one with maximum degree
$v-1$.) Perhaps the easiest way to describe $\QS(v,e)$ is to say
that it is the graph complement of $\QC(v,e^\prime)$, where
$e^\prime = \binom{v}{2}-e$.

Define the function $C(v,e)$ to be the sum of the squares of the
degree sequence of the quasi-complete graph in $\G(v,e)$, and define
$S(v,e)$ to be the sum of the squares of the degree sequence of the
quasi-star graph in $\G(v,e)$. The value of $C(v,e)$ can be computed
as follows:

Let $e=\binom{k+1}{2}-j$, with $1 \leq j \leq k$.  The degree
sequence of the \qc graph in $\G(v,e)$ is
\[
d_1= \cdots = d_{k-j} = k, \quad d_{k-j+1}= \cdots = d_{k}=k-1,
\quad d_{k+1}=k-j, \quad d_{k+2} = \cdots = d_v=0.
\]
Hence
\begin{equation} \label{eqn:Cve}
      C(v,e)  =  j(k-1)^2 + (k-j)k^2 + (k-j)^2.
\end{equation}
Since $\QS(v,e)$ is the complement of $\QC(v,e^\prime)$, it is
straightforward to show that
\begin{eqnarray} \label{eqn:SCconnection}
    S(v,e) & = & C(v,e^\prime) +(v-1)(4e-v(v-1))
\end{eqnarray}
from which it follows that, for fixed $v$, the function
$S(v,e)-C(v,e)$ is point-symmetric about the middle of the
interval $0 \leq e \leq \binom{v}{2}$. In other words,
\begin{eqnarray*}
     S(v,e)-C(v,e) & = & - \left( S(v,e^\prime)-C(v,e^\prime)
     \right).
\end{eqnarray*}
It also follows from Equation (\ref{eqn:SCconnection}) that
$\QC(v,e)$ is optimal in $\G(v,e)$ if and only if $\QS(v,e^\prime)$
is optimal in $\G(v,e^\prime)$. This allows us to restrict our
attention to values of $e$ in the interval $ [0,\binom{v}{2}/2]$ or
equivalently the interval $[\binom{v}{2}/2,\binom{v}{2}]$. On
occasion, we will do so but we will always state results for all
values of $e$.

As the midpoint of the range of values for $e$ plays a recurring
role in what follows, we denote it by
\[
m=m(v) =  \frac{1}{2} \binom{v}{2}
\]
and define $k_0=k_0(v)$ to be the integer such that
\begin{equation} \label{eqn:k0}
   \binom{k_0}{2}  \leq m <  \binom{k_0+1}{2}.
\end{equation}
To state the results of \cite{AK} we need one more notion, that of
the distance from $\binom{k_0}{2}$ to $m$. Write
\[
    b_0 =b_0(v) = m-\binom{k_0}{2}.
\]

We are now ready to summarize the results of \cite{AK}:

\begin{description}
\item[Theorem 2 \cite{AK}] $\max(v,e)$ is the larger of the two
values $C(v,e)$ and $S(v,e)$.
 \item[Theorem 3 \cite{AK}]
$\max(v,e)=S(v,e)$ if $0 \leq e <
m-\frac{v}{2}$ and $\max(v,e)=C(v,e)$ if
$m+\frac{v}{2} < e \leq \binom{v}{2}$
\item[Lemma 8 \cite{AK}]
%If $2b_0 \geq 2v-2k_0-1$,
If $2b_0 \geq k_0$, or $2v-2k_0-1  \leq 2b_0< k_0$, then
\begin{eqnarray*}
C(v,e) \leq S(v,e) & \mbox{ for all } & 0 \leq e \leq m \mbox{
and} \\
 C(v,e) \geq S(v,e) & \mbox{ for all } & m \leq e \leq
 \binom{v}{2}.
\end{eqnarray*}
%\item[Lemma 8 \cite{AK}]
%If $2b_0 < 2v-2k_0-1$,
If $2b_0 < k_0$ and $2k_0+2b_0< 2v-1$, then there exists an $R$ with
$b_0 \leq R \leq \min \left\{ v/2,k_0-b_0 \right\}$ such that
\begin{eqnarray*}
   C(v,e) \leq S(v,e) & \mbox{ for all } & 0 \leq e \leq m-R \\
   C(v,e) \geq S(v,e) & \mbox{ for all } & m-R \leq e \leq m
   \\
   C(v,e) \leq S(v,e) & \mbox{ for all } & m \leq e \leq m+R \\
   C(v,e) \geq S(v,e) & \mbox{ for all } & m+R \leq e \leq
   \binom{v}{2}.
\end{eqnarray*}
\end{description}

%\framebox{\parbox{7in}{{\tt REMOVE?} We can summarize Lemma 8  in \cite{AK} like this:  For some values of $v$, the \qc graph is optimal in $\G(v,e)$ for all small values of $e$ and the \qs graph is optimal for all large values of $e$.  (A small $e$ is one in the interval from $0$ to the middle $m$; a large value of $e$ is one in the interval from $m$ to $\binom{v}{2}$, which is the largest possible value of $e$.  But for the other values of $v$, we must divide the interval $[0,\binom{v}{2}]$ of possible values for $e$ into four subintervals, $[0,m-R]$, $[m-R,m]$, $[m,m+R]$, $[m+R, \binom{v}{2}]$.  The \qs graph is optimal if $e$ is in the first or the third interval, otherwise the \qc graph is optimal.}}

Ahlswede and Katona pose some open questions at the end of
\cite{AK}. ``Some strange number-theoretic combinatorial questions
arise.  What is the relative density of the numbers $v$ for which $R=0$
[$\max(v,e)=S(v,e)$ for all $0 \leq e < m$ and
$\max(v,e)=C(v,e)$ for all $m<e\leq \binom{v}{2}$]?"

This is the point of departure for our paper. Our first main result,
Theorem \ref{thm:main1}, strengthens Ahlswede and Katona's Theorem
2; not only does the maximum value of $P_2(G)$ occur at either the
\qs or \qc graph in $\G(v,e)$, but all optimal graphs in $\G(v,e)$
are related to the \qs or \qc graphs via their so-called diagonal
sequence.  As a result of their relationship to the \qs and \qc
graphs, all optimal graphs can be and are described in our second
main result, Theorem \ref{thm:main2}.    Our third main result,
Theorem \ref{thm:main3}, is a refinement of Lemma 8 in \cite{AK}.
Theorem \ref{thm:main3} characterizes the values of $v$ and $e$ for
which $S(v,e)=C(v,e)$ and gives an explicit expression for the value
$R$ in Lemma 8 of \cite{AK}.  Finally, the ``strange
number-theoretic combinatorial" aspects of the problem, mentioned by
Ahlswede and Katona, turn out to be Pell's Equation $y^2-2x^2=\pm
1$. Corollary \ref{cor:density}  answers  the density question posed
by Ahlswede and Katona. We have just recently learned that Wagner
and Wang \cite{WW} have independently answered this question as
well. Their approach is similar to ours, as they also find an
expression for $R$ in Lemma 8 of \cite{AK}.

Before stating the new results, we summarize work on the problem
that followed \cite{AK}.

A generalization of the problem of maximizing the sum of the squares
of the degree sequence was investigated by Katz \cite{Ka} in 1971
and R. Aharoni \cite{Aha} in 1980. Katz's problem was to maximize
the sum of the elements in $A^2$, where $A$ runs over all
$(0,1)$-square matrices of size $n$ with precisely $j$ ones.   He
found the maxima and the matrices for which the maxima are attained
for the special cases where there are $k^2$ ones or where there are
$n^2-k^2$ ones in the $(0,1)$-matrix.  Aharoni \cite{Aha} extended
Katz's results for general $j$ and showed that the maximum is
achieved at one of four possible forms for $A$.

If $A$ is a symmetric $(0,1)$-matrix, with zeros on the diagonal,
then $A$ is the adjacency matrix $A(G)$ for a graph $G$. Now let $G$
be a graph in $\G(v,e)$.  Then the adjacency matrix  $A(G)$ of $G$
is a $v \times v$ $(0,1)$-matrix with $2e$ ones. But $A(G)$
satisfies two additional restrictions: $A(G)$ is symmetric, and all
diagonal entries are zero.  However, the sum of all entries in
$A(G)^2$ is precisely $\sum d_i(G)^2$.
 Thus our problem is essentially the same as Aharoni's in that both ask for the  maximum
of the sum of the elements in $A^2$. The graph-theory problem simply
restricts the set of $(0,1)$-matrices to those with $2e$ ones that
are symmetric and have
 zeros on the diagonal.

Olpp \cite{Olp}, apparently unaware of the work of Ahlswede and
Katona, reproved the basic result that
$\max(v,e)=\max(S(v,e),C(v,e))$, but his results are stated in the
context of  two-colorings of a graph. He investigates a question of
Goodman \cite{Go1,Go2}: maximize the number of monochromatic
triangles in a two-coloring of the complete graph  with a fixed
number of vertices and a fixed number of red edges.  Olpp shows that
Goodman's problem is equivalent to finding the two-coloring that
maximizes the sum of squares of the red-degrees of the vertices.  Of
course, a two-coloring of the complete graph on $v$ vertices gives
rise to two graphs on $v$ vertices: the graph $G$ whose edges are
colored red, and its complement $G^\prime$.  So Goodman's problem is
to find the maximum value of $P_2(G)$ for $G \in \G(v,e)$.

Olpp \cite{Olp} shows that either  the \qs or the \qc graph is
optimal in $\G(v,e)$, but he does not discuss which of the two
values $S(v,e), C(v,e)$ is larger.  He leaves this question
unanswered and he does not attempt to identify all optimal graphs in
$\G(v,e)$.

In 1999, Peled, Pedreschi, and Sterbini \cite{PPS} showed that the
only possible graphs for which the maximum value is attained are the
so-called threshold graphs.   The main result in \cite{PPS} is that
all optimal graphs  are in one of six classes  of threshold graphs.
They end with the remark, ``Further questions suggested by this work
are the existence and uniqueness of the [graphs in $\G(v,e)$] in
each class, and the precise optimality conditions."

Also in 1999, Byer \cite{By} approached the problem in yet another
equivalent context: he studied the maximum number of paths of length
two over all graphs in $\G(v,e)$. Every path of length two in $G$
represents an edge in the line graph $L(G)$, so this problem is
equivalent to studying the graphs that achieve $\max(v,e)$. For each
$(v,e)$, Byer shows that there are at most six graphs in $\G(v,e)$
that achieve the maximum.  These maximal graphs come from among six
general types of graphs for which there is at most one of each type
in $\G(v,e)$. He also extended his results to the problem of finding
the maximum number of monochromatic triangles (or any other fixed
connected graph with 3 edges) among two-colorings of the complete
graph on $v$ vertices, where exactly $e$ edges are colored red.
However, Byer did not discuss how to compute $\max(v,e)$, or how to
determine when any of the six graphs is optimal.

In Section \ref{sec: statements of main results}, we have unified
some of the earlier work on this problem by using partitions,
threshold graphs, and the idea of a diagonal sequence.

\section{Statements of the main results} \label{sec: statements of
main results}
\subsection{Threshold graphs}
All optimal graphs come from a class of special graphs called
\emph{threshold} graphs.  The \qs and \qc graphs are just two among
the many threshold graphs in $\G(v,e)$. The adjacency matrix of a
threshold graph has a special form.  The upper-triangular part of
the adjacency matrix of a threshold graph is left justified and the
number of zeros in each row of the upper-triangular part of the
adjacency matrix does not decrease.  We will show  adjacency
matrices using ``$+$" for the main diagonal, an empty circle
``$\circ$" for the zero entries, and a black dot, ``$\bullet$" for
the entries equal to one.

For example, the graph $G$ whose adjacency matrix is shown in Figure
\ref{fig:example643}(a) is a threshold graph in $\G(8,13)$ with
degree sequence $(6,5,5,3,3,3,1,0)$.
%\begin{figure}[htbp]
%\begin{center}
%\includegraphics[width=2in]{EX643.pdf}
%\caption{The adjacency matrix, for the graph  $\thr(\pi) \in \G(8,13)$ corresponding to the distinct partition $\pi=(6,4,3)$}
%\label{fig:example643}
%\end{center}
%\end{figure}

By looking at the upper-triangular part of the adjacency matrix,  we
can associate the distinct partition $\pi = (6,4,3)$ of 13 with the
graph.  In general, the {\it threshold} graph $\thr(\pi) \in
\G(v,e)$ corresponding to a distinct partition $\pi = (a_0, a_1,
\ldots , a_p)$ of $e$, all of whose parts are less than $v$, is the
graph with an adjacency matrix whose upper-triangular part is
left-justified and contains $a_s$ ones in row $s$.  Thus the
threshold graphs in $\G(v,e)$ are in one-to-one correspondence with
the set of distinct partitions, $\ds(v,e)$
 of $e$ with all parts less than $v$:
 \[
 \ds(v,e) = \{ \pi=(a_0,a_1, \ldots, a_p): v>a_0>a_1> \cdots > a_p>0, \sum a_s=e \}
 \]
We denote the adjacency matrix of the threshold graph $\thr(\pi)$
corresponding to the distinct partition $\pi$ by $\adj(\pi)$.

Peled, Pedreschi, and Sterbini \cite{PPS} showed that all optimal
graphs  in a graph class $\G(v,e)$ must be threshold graphs.
 \begin{lemma}\cite{PPS}  \label{lem:OptAreThres}
If $G$ is an optimal graph in $\G(v,e)$, then $G$  is a threshold graph.
\end{lemma}
Thus we can limit the search for optimal graphs to the threshold graphs.

Actually, a much larger class of functions, including the power
functions,  $d_1^p + \cdots + d_v^p$ for $p \geq 2$, on the degrees
of a graph are maximized only at threshold graphs.  In fact, every
Schur convex function of the degrees is maximized only at the
threshold graphs.    The reason is that the degree sequences of
threshold graphs are maximal with respect to the majorization order
among all graphical sequences.  See \cite{MO} for a discussion of
majorization and Schur convex functions and \cite{MR} for a
discussion of the degree sequences of threshold graphs.

\subsection{The diagonal sequence of a threshold graph}
To state the first main theorem, we must now digress to describe the
diagonal  sequence of a threshold graph in the graph class
$\G(v,e)$.

Returning to the example in Figure \ref{fig:example643}(a)
corresponding to the distinct partition  $\pi = (6,4,3) \in
\ds(8,13)$, we superimpose diagonal lines on the adjacency matrix
$\adj(\pi)$ for the threshold graph $\thr(\pi)$ as shown in Figure
\ref{fig:example643}(b).

\begin{figure}[h]
\begin{center}
\vspace{.2in} \centerline {
\epsfig{file=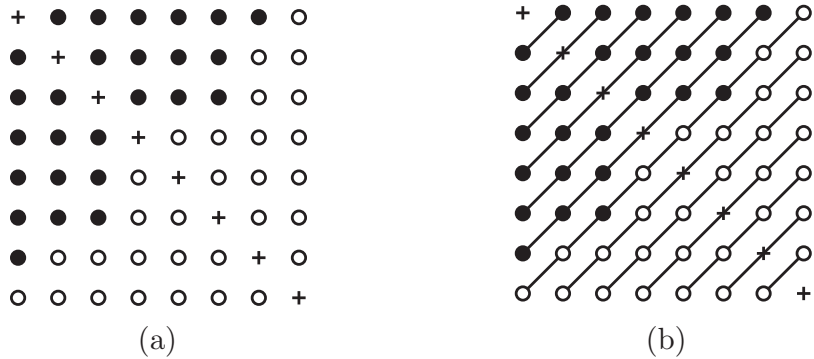,width=4.2in}
 } \vspace{.2in}

\caption{The adjacency matrix, $\adj(\pi)$, for the threshold graph
in $\G(8,13)$ corresponding to the distinct partition $\pi=(6,4,3)
\in \ds(8,13)$ with diagonal sequence
$\delta(\pi)=(1,1,2,2,3,3,1)$.}
\label{fig:example643}
%\label{fig:adjexample}
\end{center}
\end{figure}

The number of black dots in the upper triangular part of the
adjacency  matrix on each of the diagonal lines is called the {\it
diagonal sequence} of the partition $\pi$ (or of the threshold graph
$\thr(\pi)$).  The diagonal sequence for $\pi$ is denoted by
$\delta(\pi)$ and for   $\pi=(6,4,3)$ shown in Figure
\ref{fig:example643}, $\delta(\pi)=(1,1,2,2,3,3,1)$.  The value of
$P_2(\thr(\pi))$ is determined by the diagonal sequence of $\pi$.
\begin{lemma} \label{lem:diagonal}
Let $\pi$ be a distinct partition in $\ds(v,e)$ with diagonal
sequence  $\delta(\pi)=(\delta_1, \ldots, \delta_t)$.  Then
$P_2(\thr(\pi))$ is the dot product
\[
P_2(\thr(\pi))=2 \delta(\pi) \cdot (1,2,3, \ldots, t)=2\sum_{i=1}^t i\delta_i.
\]
\end{lemma}
For example, if $\pi=(6,4,3)$ as in Figure \ref{fig:example643},
then
\[
P_2(\thr(\pi)) = 2(1,1,2,2,3,3,1) \cdot (1,2,3,4,5,6,7) = 114,
\]
which equals the sum of squares of the degree sequence
$(6,5,5,3,3,3,1)$ of  the graph $\thr(\pi)$.

Theorem 2 in \cite{AK} guarantees that one (or both) of the graphs
$\QS(v,e), \QC(v,e)$ must be optimal in $\G(v,e)$.  But there may be
other optimal graphs in $\G(v,e)$, as the next example shows.

%{\tt MICHAEL: I'M GOING TO REMOVE THIS EXAMPLE FROM THE NEXT VERSION.  LET ME KNOW IF YOU THINK IT SHOULD STAY IN.
%
%The quasi-star graph $\QS(15,46)$ in $\G(15,46)$ is optimal and its
%degree sequence is $(14,14,14,10,4,4,4,4,4,4,4,3,3,3,3)$. The sum
%of the squares of the degrees is 836. The graph $G_1$ in
%$\G(15,46)$ with degree sequence
%$(13,13,13,13,4,4,4,4,4,4,4,4,4,4,0)$ is also optimal but is neither the \qs not the \qc graph in $\G(15,46)$.  However, by removing one of the isolated vertices in $G_1$, we obtain the \qs graph in  $\G(14,46)$.}
%

The quasi-complete graph $\QC(10,30)$, which corresponds to the
distinct partition  $(8,7,5,4,3,2,1)$  is optimal in $\G(10,30)$.
The threshold graph $G_2$, corresponding to the distinct partition
$(9,6,5,4,3,2,1)$ is also optimal in $\G(10,30)$, but  is neither
\qs in $\G(10,30)$ nor \qc in  $\G(v,30)$ for any $v$.  The
adjacency matrices for these two graphs are shown in Figure
\ref{fig:adjequivgraphs}.   They have the same diagonal sequence
$\delta=(1,1,2,2,3,3,4,4,4,2,2,1,1)$ and both are optimal.

\begin{figure}[htbp]
\begin{center}
\epsfig{file=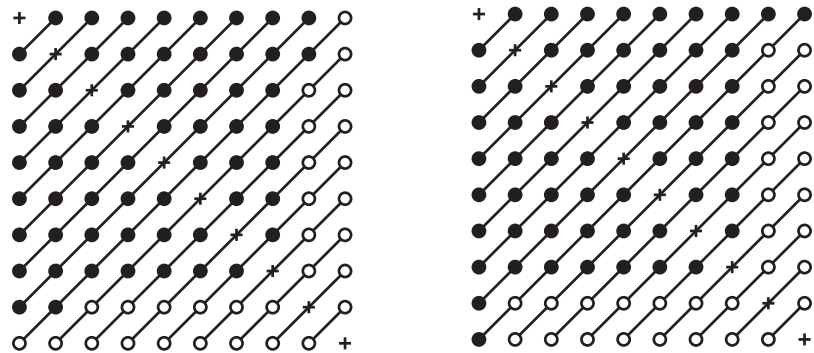,width=4.2in} \caption{Adjacency matrices
for two optimal graphs in $\G(10,30)$,
$\QC(10,30)=\thr(8,7,5,4,3,2,1)$ and $\thr(9,6,5,4,3,2,1)$, having
the same diagonal sequence $\delta=(1,1,2,2,3,3,4,4,4,2,2,1,1)$}
\label{fig:adjequivgraphs}
\end{center}
\end{figure}

We know that either the \qs or the \qc graph in $\G(v,e)$ is
optimal and that any  threshold graph with the same diagonal
sequence as an optimal graph is also optimal.  In fact, the converse
is also true.  Indeed, the relationship between the optimal graphs
and the \qs and \qc graphs in a graph class $\G(v,e)$ is described
in our first main theorem.

\begin{theorem} \label{thm:main1}
Let $G$ be an optimal graph in $\G(v,e)$.  Then $G=\thr(\pi)$ is  a
threshold graph for some partition $\pi \in \ds(v,e)$ and  the
diagonal sequence  $\delta(\pi)$ is equal to the diagonal sequence
of either the \qs graph or  the \qc graph in $\G(v,e)$.
 \end{theorem}

Theorem \ref{thm:main1}  is stronger than Lemma 8 \cite{AK} because
it characterizes {\it all} optimal graphs in $\G(v,e)$.  In Section
\ref{sec:optimalgraphs} we describe all optimal graphs in detail.

\subsection{Optimal graphs} \label{sec:optimalgraphs}
Every optimal graph in $\G(v,e)$ is a threshold graph, $\thr(\pi)$,
corresponding to a partition $\pi$ in $\ds(v,e)$.  So we extend the
terminology and  say that the partition $\pi$ is {\it optimal} in
$\ds(v,e)$, if its threshold graph $\thr(\pi)$ is optimal in
$\G(v,e)$. We say that the partition $\pi \in \ds(v,e)$ is the {\it
\qs partition}, if  $\thr(\pi)$ is the \qs graph in $\G(v,e)$.
Similarly, $\pi \in \ds(v,e)$ is the {\it \qc partition}, if
$\thr(\pi)$ is the \qc graph in $\G(v,e)$.

We now describe the \qs and \qc partitions in $\ds(v,e)$.

First, the \qc graphs.  Let $v$ be a positive integer and $e$ an
integer such that $0 \leq e \leq \binom{v}{2}$.  There exists unique
integers $k$ and $j$ such that
\[
e=\binom{k+1}{2}-j \mbox{ and } 1 \leq j \leq k.
\]
The partition
\[
\pi(v,e,\qcc): = (k,k-1, \ldots , j+1,j-1, \ldots , 1)=(k,k-1, \ldots, \widehat{j}, \ldots,2,1)
\]
corresponds to the \qc threshold graph $\QC(v,e)$ in $\ds(v,e)$.
The symbol $\widehat{j}$ means that $j$ is missing.

To describe the \qs partition $\pi(v,e,\qss)$ in $\ds(v,e)$, let
$k^\prime, j^\prime$ be the unique integers such that
\[
e=\binom{v}{2}-\binom{k^\prime+1}{2}+j^\prime \mbox{ and } 1 \leq
j^\prime \leq k^\prime.
\]
Then the partition
\[
\pi(v,e,\qss)=(v-1, v-2, \ldots,  k^\prime+1, j^\prime)
\]
corresponds to the \qs graph $\QS(v,e)$ in $\ds(v,e)$.

In general, there may be many partitions with the same diagonal
sequence as $\pi(v,e,\qcc)$ or $\pi(v,e,\qss)$.  For example, if
$(v,e)=(14,28)$, then $\pi(14,28,\qcc)=(7,6,5,4,3,2,1)$ and all of
the partitions in Figure \ref{fig:examplev14e28} have the same
diagonal sequence, $\delta=(1,1,2,2,3,3,4,3,3,2,2,1,1)$.
\begin{figure}[htbp]
\epsfig{file=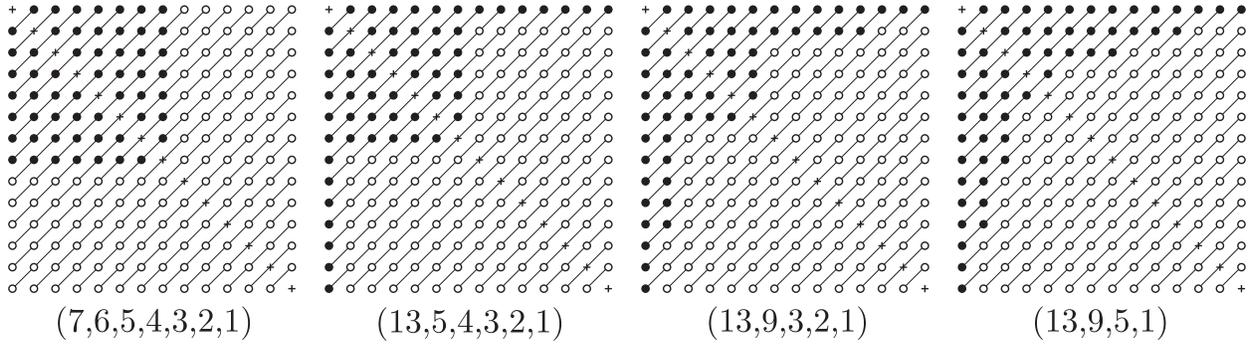,width=6.5in}
%\begin{tabular}{cccc}
%\includegraphics[width=1.5in]{QCC14E28.pdf}&
%\includegraphics[width=1.5in]{QCC14E28A.pdf}&
%\includegraphics[width=1.5in]{QCC14E28B.pdf}&
%\includegraphics[width=1.5in]{QCC14E28C.pdf}\\
%$(7,6,5,4,3,2,1)$&$(13,5,4,3,2,1)$&$(13,9,3,2,1)$&$(13,9,5,1)$
%\end{tabular}
\caption{Four partitions with the same diagonal sequence as
$\pi(14,28,\qcc)$} \label{fig:examplev14e28}
\end{figure}
But none of the threshold graphs corresponding to the partitions  in
Figure \ref{fig:examplev14e28} is optimal.  Indeed, if the \qc graph
is optimal in $\ds(v,e)$, then there are  at most three partitions
in $\ds(v,e)$ with the same diagonal sequence as the \qc graph.  The
same is true for the \qs partition.  If the \qs partition is optimal
in $\ds(v,e)$,  then there are at most three partition in $\ds(v,e)$
having the same diagonal sequence as the \qs partition.  As a
consequence, there are at most six optimal partitions in $\ds(v,e)$
and so at most six optimal graphs in $\G(v,e)$.  Our second main
result, Theorem \ref{thm:main2}, entails Theorem \ref{thm:main1}; it
describes the optimal partitions in $\G(v,e)$  in detail. The six
partitions described in Theorem \ref{thm:main2} correspond to the
six graphs determined by Byer in \cite{By}. However, we give precise
conditions to determine when each of these partitions is optimal.

\begin{theorem} \label{thm:main2}
Let  $v$ be a positive integer and $e$ an integer such that $0 \leq
e \leq \tbinom{v}{2}$.  Let $k,k^\prime,j,j^\prime$ be the unique
integers satisfying
\[
e = \binom{k+1}{2}-j, \text{ with } 1 \leq j \leq k,
\]
and
\[
e = \binom{v}{2}-\binom{k^\prime+1}{2}+j^\prime, \text{ with } 1 \leq j^\prime \leq k^\prime.
\]
Then every optimal partition $\pi$ in $\ds(v,e)$ is one of the
following six partitions:
\begin{description}
\item[1.1] $\pi_{1.1}=(v-1,v-2,\ldots ,k^\prime+1,j^\prime)$, the \qs partition for $e$,
\item[1.2] $\pi_{1.2}=(v-1,v-2,\ldots ,\widehat{2k^\prime-j^\prime-1},\ldots
,k^\prime-1)$, if $k^\prime+1\leq 2k^\prime-j^\prime-1 \leq v-1$,
\item[1.3] $\pi_{1.3}=(v-1,v-2,\ldots ,k^\prime+1,2,1)$, if $j^\prime=3$ and $v\geq4$,
\item [2.1] $\pi_{2.1}=(k,k-1, \ldots, \widehat{j} , \ldots,2,1)$, the \qc partition for $e$,
\item [2.2] $\pi_{2.2}=(2k-j-1,k-2,k-3, \ldots 2,1)$, if $k+1 \leq 2k-j-1 \leq v-1$,
\item [2.3] $\pi_{2.3}=(k,k-1, \ldots, 3)$, if $j=3$ and $v\geq4$.
\end{description}

On the other hand, partitions $\pi_{1.1}$ and $\pi_{2.1}$ always
exist  and at least one of them is optimal. Furthermore, $\pi_{1.2}$ and $\pi_{1.3}$ (if they exist) have the same diagonal
sequence as $\pi_{1.1}$, and if $S(v,e)\geq C(v,e)$, then they are all optimal.
Similarly, $\pi_{2.2}$ and $\pi_{2.3}$ (if they exist)
have the same diagonal sequence as $\pi_{2.1}$, and if $S(v,e)\leq C(v,e)$, then
they are all optimal.
\end{theorem}

%Let $v$ be a positive integer and $e$ an integer such that $0 \leq e \leq \tbinom{v}{2}$.  Let $k,k^\prime,j,j^\prime$ be the unique integers satisfying
%\[
%e = 1+2+ \cdots + k -j
%=\binom{k+1}{2}-j, \text{ with } 1 \leq j \leq k,
%\]
%and
%\[
%e = (v-1) + (v-2) + \cdots + (k^\prime+1) +j^\prime
%=\binom{v}{2}-\binom{k^\prime+1}{2}+j^\prime, \text{ with } 1 \leq j^\prime \leq k^\prime.
%\]
%\begin{enumerate}
%\item If the \qs partition is optimal in $\ds(v,e)$ and the \qc partition is not optimal, then the only optimal partitions $\pi$ in $\ds(v,e)$ are:

%\begin{description}
%\item[1.1] $\pi=(v-1,v-2,\ldots ,k^\prime+1,j^\prime)$, the \qs partition for $e$,
%\item[1.2] $\pi=(v-1,v-2,\ldots ,\widehat{2k^\prime-j^\prime-1},\ldots
%,k^\prime-1)$, if $2k^\prime-j^\prime \leq v$ and $j^\prime \leq k^\prime-2$,
%\item[1.3] $\pi=(v-1,v-2,\ldots ,k^\prime+1,2,1)$, if $j^\prime=3$.
%\end{description}
%Furthermore, the partitions in 1.1, 1.2, and 1.3 are distinct.
%\item If the \qc partition is optimal in $\ds(v,e)$ and the \qs partition is not optimal, then the only optimal partitions $\pi$ in $\ds(v,e)$ are:

%\begin{description}
%\item[2.1] $\pi=(k,k-1, \ldots, \widehat{j}, \ldots, 2,1)$, the \qc partition for $e$,
%\item[2.2] $\pi=(2k-j-1, k-2, k-3, \ldots,2,1)$, if $2k-j \leq v$ and $j \leq k-2$
%\item[2.3] $\pi=(k,k-1, \ldots, 3)$, if $j=3$.
%\end{description}

%Furthermore, the partitions in 2.1, 2.2, and 2.3 are distinct.
%\item If both the \qs partition and the \qs partition are optimal in $\ds(v,e)$, then the only optimal partitions $\pi$ in $\ds(v,e)$ are:

%
%\end{enumerate}
%\end{theorem}

A few words of explanation  are in order regarding the notation for
the optimal partitions in Theorem \ref{thm:main2}. If $k^\prime=v$,
then $j^\prime=v,e=0$, and $\pi_{1.1}=\emptyset$. If $k^\prime=v-1$,
then $e=j^\prime \leq v-1$, and $\pi_{1.1}=(j^\prime)$; further, if
$j^\prime=3$, then $\pi_{1.3}=(2,1)$. In all other cases $k^\prime
\leq v-2$ and then $\pi_{1.1}$, $\pi_{1.2}$, and $\pi_{1.3}$ are
properly defined.

If $j^\prime = k^\prime$ or $j^\prime= k^\prime -1$, then both
partitions in 1.1 and 1.2 would be equal to $(v-1,v-2, \ldots,
k^\prime)$ and $(v-1,v-2, \ldots, k^\prime+1, k^\prime-1)$
respectively. So the condition $k^\prime+1 \leq 2
k^\prime-j^\prime-1$ merely ensures that $\pi_{1.1}\neq \pi_{1.2}$. A similar remark holds for the partitions in 2.1 and 2.2. By
definition the partitions $\pi_{1.1}$ and $\pi_{1.3}$ are always
distinct; the same holds for partitions $\pi_{2.1}$ and $\pi_{2.3}$.
In general the partitions $\pi_{i.j}$ described in items 1.1-1.3 and
2.1-2.3 (and their corresponding threshold graphs) are all
different. All the exceptions are illustrated in Figure
\ref{fig:hexagon} and are as follows: For any $v$, if $e\in
\{0,1,2\}$ or $e^{\prime }\in \{0,1,2\}$ then $\pi_{1.1}=\pi_{2.1}$.
For any $v\geq4$, if $e=3$ or $ e^\prime=3$, then
$\pi_{1.3}=\pi_{2.1}$ and $\pi_{1.1}=\pi_{2.3}$. If $(v,e)=(5,5)$
then $\pi_{1.1}=\pi_{2.2}$ and $\pi_{1.2}=\pi_{2.1}$. Finally, if
$(v,e)=(6,7)$ or $(7,12)$, then $\pi_{1.2}=\pi_{2.3}$. Similarly, if
$(v,e)=(6,8)$ or $(7,9)$, then $\pi_{1.3}=\pi_{2.2}$. For $v\geq8$ and $4\leq e \leq \tbinom{v}{2}-4$, all the partitions $\pi_{i.j}$ are pairwise distinct (when they exist).

\begin{figure}[h]
\begin{center}
\epsfig{file=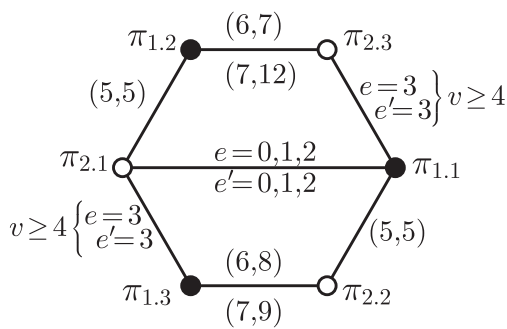,width=2.6in} \caption{Instances of pairs
$(v,e)$ where two partitions $\pi_{i.j}$ coincide}
\label{fig:hexagon}
\end{center}
\end{figure}

In the next section, we determine the pairs $(v,e)$
having a prescribed number of optimal partitions (and hence graphs)
in $\G(v,e)$.

\subsection{Pairs $(v,e)$ with a prescribed number of optimal partitions.}

In principle, a given pair $(v,e)$, could have between one and six
optimal partitions. It is easy to see that there are infinitely many
pairs $(v,e)$ with only one optimal partition (either the quasi-star
or the quasi-complete). For example the pair $(v, \tbinom{v}{2})$
only has the quasi-complete partition. Similarly, there are
infinitely many pairs with exactly two optimal partitions and this
can be achieved in many different ways. For instance if
$(v,e)=(v,2v-5)$ and $v\geq 9$, then $k^\prime=v-2$,
$j^\prime=v-4>3$, and $S(v,e)>C(v,e)$ (c.f. Corollary
\ref{cor:strictuniform}). Thus only the partitions $\pi_{1.1}$ and
$\pi_{1.2}$ are optimal. The interesting question is the existence
of pairs with 3,4,5, or 6 optimal partitions.

Often, both partitions $\pi_{1.2}$ and $\pi_{1.3}$ in Theorem
\ref{thm:main2} exist for the same pair $(v,e)$; however it turns
out that this almost never happens when they are optimal partitions.
More precisely,

\begin{theorem} \label{thm:family5_6}
If $\pi_{1.2}$ and $\pi_{1.3}$ are optimal partitions then
$(v,e)=(7,9)$ or $(9,18)$. Similarly, if $\pi_{2.2}$ and $\pi_{2.3}$
are optimal partitions then $(v,e)=(7,12)$ or $(9,18)$. Furthermore,
the pair $(9,18)$ is the only one with six optimal partitions, there
are no pairs with five. If there are more than two optimal
partitions for a pair $(v,e)$, then $S(v,e)=C(v,e)$, that is, both
the quasi-complete and the quasi-star partitions must be optimal.
\end{theorem}

In the next two results, we describe two infinite families of
partitions in $\ds(v,e)$, and hence graph classes $\G(v,e)$, for
which there are exactly three (four) optimal partitions. The fact
that they are infinite is proved in Section \ref{sec:Pell}.
\begin{theorem} \label{thm:family3}
Let $v>5$ and $k$ be positive integers that satisfy the Pell's Equation
\begin{equation} \label{eqn:family3}
(2v-3)^2-2(2k-1)^2=-1
\end{equation}
and let $e=\binom{k}{2}$.  Then (using the notation of Theorem
\ref{thm:main2}), $j=k$,  $k^\prime=k+1$, $j^\prime=2k-v+2$, and
there are exactly three optimal partitions in $\ds(v,e)$, namely
\begin{eqnarray*}
\pi_{1.1}&=& (v-1,v-2, \ldots , k+2, 2k-v+2)\\
\pi_{1.2}&=&(v-2,v-3, \ldots, k)\\
\pi_{2.1} &=& (k-1,k-2,\ldots,2,1).
\end{eqnarray*}
The partitions $\pi_{1.3},\pi_{2.2}$, and $\pi_{2.3}$ do not exist.
\end{theorem}

\begin{theorem} \label{thm:family4}
Let $v>9$ and $k$ be positive integers that satisfy the Pell's Equation
\begin{equation} \label{eqn:family4}
(2v-1)^2-2(2k+1)^2=-49
\end{equation}
and $e=m=\frac{1}{2}\binom{v}{2}$.  Then (using the notation of Theorem \ref{thm:main2}), $j=j^\prime=3$, $k=k^\prime$, and there are exactly four optimal partitions in $\ds(v,e)$, namely
\begin{eqnarray*}
\pi_{1.1}&=& (v-1,v-2, \ldots , k+1, 3)\\
\pi_{1.3}&=&(v-1,v-2, \ldots, k+1,2,1)\\
\pi_{2.1} &=& (k-1,k-2,\ldots,4,2,1)\\
\pi_{2.3} &=& (k-1,k-2,\ldots,4,3).
\end{eqnarray*}
The partitions $\pi_{1.2}$ and $\pi_{2.2}$ do not exist.
\end{theorem}

\subsection{Quasi-star versus quasi-complete} \label{sec:QSvsQC}
In this section, we compare $S(v,e)$ and $C(v,e)$. The main result
of the section, Theorem \ref{thm:main3}, is a theorem very much like
Lemma 8 of \cite{AK}, with the addition that our results give
conditions for equality of the two functions.

If $e=0,1,2,3$, then $S(v,e)=C(v,e)$ for all $v$.  Of course, if
$e=0$, $e=1$ and $v \geq 2$, or $e \leq 3$ and $v=3$, there is only
one graph in the graph class $\G(v,e)$.  If $e=2$ and $v \geq 4$,
then there are two graphs in the graph class $\G(v,2)$: the path $P$
and the partial matching $M$, with degree sequences $(2,1,1)$ and
$(1,1,1,1)$, respectively. The path is optimal as $P_2(P)=6$ and
$P_2(M)=4$. But the path is both the \qs and the \qc graph in
$\G(v,2)$.  If $e=3$ and $v \geq 4$, then the \qs graph has degree
sequence $(3,1,1,1)$ and the \qc graph is a triangle with degree
sequence $(2,2,2)$. Since $P_2(G)=12$ for both of these graphs, both
are optimal. Similarly, $S(v,e)=C(v,e)$ for $e=\tbinom{v}{2}-j$ for
$j=0,1,2,3$.

Now, we consider the cases where $4 \leq e \leq \binom{v}{4}-4$.
Figures \ref{fig:V25}, \ref{fig:V15}, \ref{fig:V17}, and
\ref{fig:V23} show the values of the difference $S(v,e)-C(v,e)$.
When the graph is above the horizontal axis, $S(v,e)$ is strictly
larger than $C(v,e)$ and so the \qs graph is optimal and the \qc is
not optimal.  And when the graph is on the horizontal axis,
$S(v,e)=C(v,e)$ and both the \qs and the \qc graph are optimal.
Since the function $S(v,e)-C(v,e)$ is central symmetric, we shall
consider only the values of $e$ from $4$ to the midpoint, $m$,  of
the interval $[0,\binom{v}{2}]$.

Figure \ref{fig:V25} shows that $S(25,e) > C(25,e)$ for all values
of $e$: $4 \leq e < m=150$.  So, when $v=25$, the \qs graph is
optimal for $0 \leq e < m=150$ and the \qc graph is not optimal. For
$e=m(25)=150$, the \qs and the \qc graphs are both optimal.
\begin{figure}[htbp]
\begin{center}
%\rotatebox{90}{\resizebox{!}{4in}{%
\epsfig{file=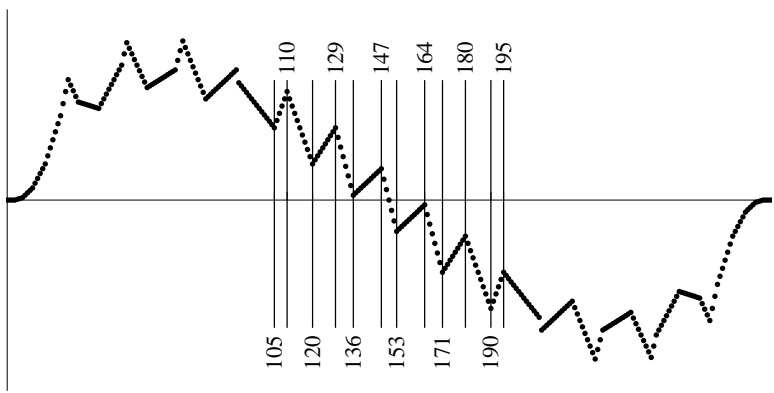,width=4in}
\caption{$S(25,e)-C(25,e)> 0$ for $4 \leq e < m=150$
% and $q_0(25)=26>0$
}
\label{fig:V25}
\end{center}
\end{figure}

Figure \ref{fig:V15} shows that $S(15,e)>C(15,e)$ for $4 \leq e<45$
and $45 < e \leq m=52.5$.  But $S(15,45)=C(15,45)$.  So the \qs
graph is optimal and the \qc graph is not optimal for all $0 \leq e
\leq 52$ except for $e=45$.  Both the \qs and the \qc graphs are
optimal in $\G(15,45)$.
\begin{figure}[htbp]
\begin{center}
%\rotatebox{90}{\resizebox{!}{4in}{%
\epsfig{file=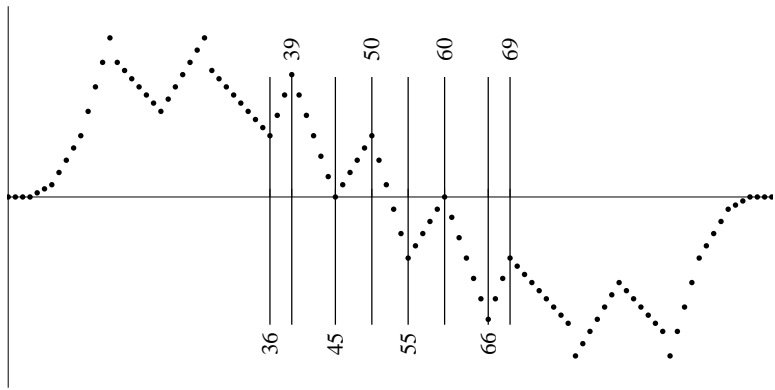,width=4in}
\caption{$S(15,e)-C(15,e)> 0$ for $4 \leq e <45$ and for $45 < e \leq m=52.5$
% and $q_0(15)=$
}
\label{fig:V15}
\end{center}
\end{figure}

Figure \ref{fig:V17} shows that $S(17,e)>C(17,e)$ for $4 \leq e<63$,
$S(17,64)=C(17,64)$,  $S(17,e)<C(17,e)$ for $65 \leq e < m=68$, and
$S(17,68)=C(17,68)$.
\begin{figure}[htbp]
\begin{center}
%\rotatebox{90}{\resizebox{!}{4in}{%
\epsfig{file=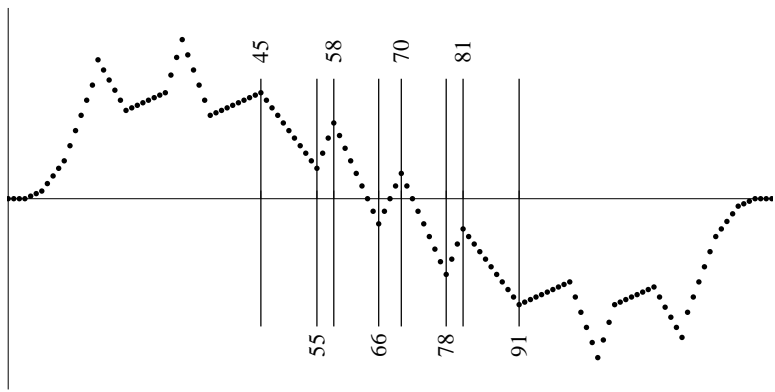,width=4in}
\caption{$S(17,e)-C(17,e)>0$ for $4 \leq e \leq 63$ and $65 \leq e
<m=68$
% and $q_0(17)=$
}
\label{fig:V17}
\end{center}
\end{figure}

Finally, Figure \ref{fig:V23} shows that $S(23,e)>C(23,e)$ for $4
\leq e \leq 119$, but $S(23,e)=C(23,e)$ for $120 \leq e \leq
m=126.5$.
\begin{figure}[htbp]
\begin{center}
%\rotatebox{90}{\resizebox{!}{4in}{%
\epsfig{file=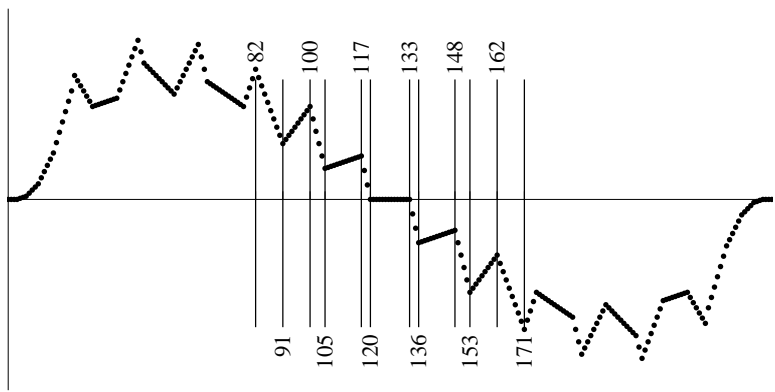,width=4in}
\caption{$S(23,e)-C(23,e)> 0$ for $4 \leq e \leq119$,
$S(23,e)=C(23,e)$ for $120 \leq e < m=126.5$
% and $q_0(23)=$
}
\label{fig:V23}
\end{center}
\end{figure}

These four examples exhibit the types of behavior of the function
$S(v,e)-C(v,e)$, for fixed $v$.   The main thing that determines this behavior  is the quadratic function

\[
 q_0(v)  :=  \frac{1}{4} \left( 1-2(2k_0-3)^2+ (2v-5)^2 \right).
 \]
(The integer $k_0=k_0(v)$ depends on $v$.)
For example, if $q_0(v)>0$, then $S(v,e)-C(v,e) \geq 0$ for all values of $e<m$.  To describe the behavior of  $S(v,e)-C(v,e)$ for $q_0(v)<0$, we need to define
\[
R_0=R_0(v)  =   \frac{8(m-e_0)(k_0-2)}{-1-2(2k_0-4)^2+(2v-5)^2},
\]
where
\[
e_0 = e_0(v)=\binom{k_0}{2}=m-b_0
\]
Our third main theorem is the following:

\begin{theorem} \label{thm:main3}
Let $v$ be a positive integer
   \begin{enumerate}
     \item If $q_0(v) > 0$, then
     \begin{eqnarray*}
         S(v,e) \geq C(v,e) & \mbox{ for all } & 0 \leq e \leq
         m \mbox{ and } \\
          S(v,e) \leq C(v,e) & \mbox{ for all } & m \leq e \leq
         \tbinom{v}{2}.
     \end{eqnarray*}
     $S(v,e)=C(v,e)$ if and only if $e,e' \in \{ 0,1,2,3,m \}$, or
      $e,e^\prime=e_0$ and $(2v-3)^2-2(2k_0-3)^2=-1,7$.
     \item If $q_0(v) <0$, then
     \begin{eqnarray*}
           C(v,e) \leq S(v,e) & \mbox{ for all } & 0 \leq e \leq m-R_0 \\
           C(v,e) \geq S(v,e) & \mbox{ for all } & m-R_0 \leq e \leq m
                \\
           C(v,e) \leq S(v,e) & \mbox{ for all } & m  \leq e \leq m+R_0 \\
           C(v,e) \geq S(v,e) & \mbox{ for all } & m+R_0 \leq e \leq
           \tbinom{v}{2}.
     \end{eqnarray*}
     $S(v,e)=C(v,e)$ if and only if $e,e^\prime \in \{
     0,1,2,3,m-R_0,m \}$
     \item If $q_0(v)=0$, then
     \begin{eqnarray*}
         S(v,e) \geq C(v,e) & \mbox{ for all } & 0 \leq e \leq
         m \mbox{ and } \\
          S(v,e) \leq C(v,e) & \mbox{ for all } & m \leq e \leq
         \tbinom{v}{2}.
     \end{eqnarray*}
     $S(v,e)=C(v,e)$ if and only if $e,e' \in \{ 0,1,2,3,e_0,..., m \}$
   \end{enumerate}
\end{theorem}

The  conditions in  Theorem \ref{thm:main3} involving the quantity
$q_0(v)$ simplify and refine the conditions in \cite{AK} involving
$k_0$ and $b_0$. The condition $2b_0 \geq k_0$ in Lemma 8 of
\cite{AK} can be removed and the result  restated in terms of the
sign of the quantity $2k_0+2b_0-(2v-1)=-2q_0(v)$. While \cite{AK}
considers only the two cases $q_0(v) \leq 0$ and $q_0(v)>0$, we
analyze the case $q_0(v)=0$  separately.

It is apparent from Theorem \ref{thm:main3} that $S(v,e) \geq
C(v,e)$ for $0 \leq e \leq m-\alpha v$ if $\alpha>0$ is large
enough.  Indeed,  Ahlswede and Katona \cite[Theorem 3]{AK} show this
for $\alpha = 1/2$, thus establishing an inequality that holds for
all values of $v$ regardless of the sign of $q_0(v)$.  We improve
this result and show that the inequality holds when $\alpha =
1-\sqrt{2}/2 \approx 0.2929$.

\begin{corollary} \label{cor:uniform}
   Let $\alpha=1-\sqrt{2}/2$. Then $S(v,e) \geq C(v,e)$ for all $0
   \leq e \leq m-\alpha v$ and $S(v,e) \leq C(v,e)$ for all $m+\alpha v
   \leq e \leq \binom{v}{2}$. Furthermore, the constant $\alpha$
   cannot be replaced by a smaller value.
\end{corollary}

Theorem 3 in \cite{AK} can be improved in another way.   The inequalities are actually strict.

\begin{corollary} \label{cor:strictuniform}
   $S(v,e)>C(v,e)$ for $4 \leq e < m-v/2$ and
   $S(v,e)<C(v,e)$ for $m+v/2 <e \leq \binom{v}{2}-4$.
\end{corollary}

\subsection{Asymptotics and density} \label{sec:density}
We now turn to the questions asked in \cite{AK}:

What is the relative density of the positive integers $v$ for which
$\max(v,e)=S(v,e)$ for $0 \leq e < m$? Of course,
$\max(v,e)=S(v,e)$ for $0 \leq e \leq m$ if and only if
$\max(v,e)=C(v,e)$ for $m\leq e \leq \binom{v}{2}$.

\begin{corollary} \label{cor:density}
Let $t$ be a positive integer and let $n(t)$ denote the number of integers $v$ in the interval $[1,t]$ such that
\[
\max(v,e)=S(v,e),
\]
for all $0 \leq e \leq m$.
Then
\[
\lim_{t \rightarrow \infty} \frac{n(t)}{t} = 2-\sqrt{2} \approx
0.5858.
\]
\end{corollary}

\subsection{Piecewise linearity of $S(v,e)-C(v,e)$} \label{sec:piecelinearity}
The diagonal sequence for a threshold graph helps explain the
behavior of the difference $S(v,e)-C(v,e)$ for fixed $v$ and $0 \leq
e \leq \binom{v}{2}$. From Figures  \ref{fig:V25}, \ref{fig:V15},
\ref{fig:V17}, and \ref{fig:V23}, we see that $S(v,e)-C(v,e)$,
regarded as a function of $e$, is piecewise linear and the ends of
the intervals on which the function is linear occur at
$e=\binom{j}{2}$ and $e=\binom{v}{2}-\binom{j}{2}$ for $j =1,2,
\ldots, v$. We prove this fact in Lemma \ref{lem:linearity}. For now, we present an example.

Take $v=15$, for example.  Figure \ref{fig:V15} shows linear
behavior on the intervals $[36,39]$, $[39,45]$, $[45,50]$, $[50,55]$,
$[55,60]$, $[60,66]$, and $[66,69]$.  There are 14 binomial
coefficients $\binom{j}{2}$ for $2 \leq j \leq 15$:
\[
1,3,6,10,15,21,28,36,45,55,66,78,91,105.
\]
The complements with respect to $\binom{15}{2}=105$ are
\[
104, 102, 99, 95, 90, 84, 77, 69, 60, 50, 39, 27, 14, 0.
\]
The union of these two sets of integers coincide with the end points
for the intervals on which $S(15,e)-C(15,e)$ is linear.  In this
case, the function is linear on the 27  intervals with end points:
\[
0,1,3,6,10,14,15,21,27,28,36,39,45,50,55,60,66,69,77,78,84,90,91,95,99,102,104,105.
\]
These special values of $e$ correspond to special types of \qs and
\qc graphs.

If $e=\binom{j}{2}$, then the \qc graph $\QC(v,e)$ is the sum of a
complete graph on $j$ vertices and $v-j$ isolated vertices.  For
example, if $v=15$ and $j=9$, and $e=\binom{9}{2}=36$, then the
upper-triangular part of the adjacency matrix for $\QC(15,21)$ is
shown on the left in Figure \ref{fig:V15E36789}. And if
$e=\binom{v}{2}-\binom{j}{2}$, then the \qs graph $\QS(v,e)$ has $j$
dominant vertices and none of the other $v-j$ vertices are adjacent
to each other.  For example, the lower triangular part of the
adjacency matrix for the \qs graph with $v=15$, $j=12$, and
$e=\binom{14}{2}-\binom{12}{2}=39$, is shown on the right in Figure
\ref{fig:V15E36789}.

\begin{figure}[t]
\begin{tabular}{llll}
\qc partition\\
$\scriptstyle \pi=(8,7,6,5,4,3,2,1)$ &$\scriptstyle
\pi=(9,7,6,5,4,3,2,1)$ &$\scriptstyle \pi=(9,8,6,5,4,3,2,1)$
&$\scriptstyle \pi=(9,8,7,5,4,3,2,1)$\\
\hline
\qs partition\\
$\scriptstyle \pi=(14,13,9)$ &$\scriptstyle \pi=(14,13,10)$
&$\scriptstyle \pi=(14,13,11)$
&$\scriptstyle \pi=(14,13,12)$\\
\hline
%\end{tabular}
%\epsfig{file=CSV15E36_2.eps,width=6.5in}
\includegraphics[width=1.4in]{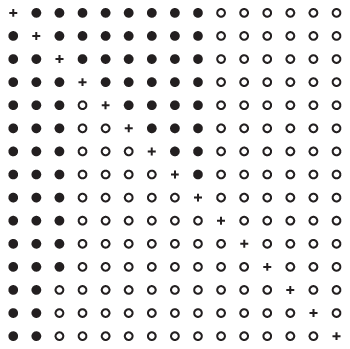}&
\includegraphics[width=1.4in]{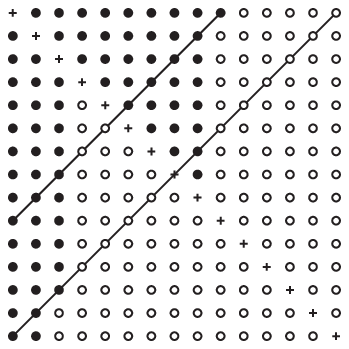}&
\includegraphics[width=1.4in]{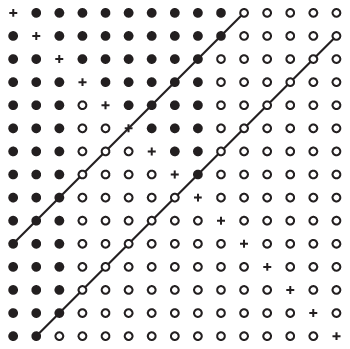}&
\includegraphics[width=1.4in]{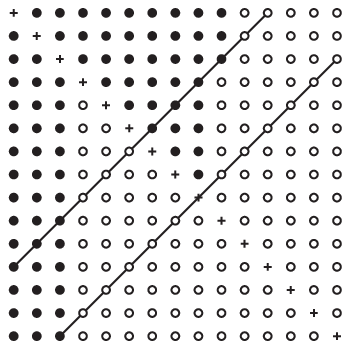}\\
%\includegraphics[width=1.4in]{CSV15E36.pdf}&
%\includegraphics[width=1.4in]{CSV15E37.pdf}&
%\includegraphics[width=1.4in]{CSV15E38.pdf}&
%\includegraphics[width=1.4in]{CSV15E39.pdf}\\
%\begin{tabular}{llll}
\hline
%\qs graphs&&&\\
%\includegraphics[width=1.4in]{SV15E36.pdf}&
%\includegraphics[width=1.4in]{SV15E37.pdf}&
%\includegraphics[width=1.4in]{SV15E38.pdf}&
%\includegraphics[width=1.4in]{SV15E39.pdf}\\

$e=36$&$e=37$&$e=38$&$e=39$
\end{tabular}

\caption{Adjacency matrices for \qc and \qs graphs with $v=15$ and
$36 \leq e \leq 39$} \label{fig:V15E36789}
\end{figure}

As additional dots are added to the adjacency matrices for the \qc
graphs with $e=37,38,39$, the value of $C(15,e)$ increases by
$18,20,22$.  And the value of $S(15,e)$ increases by $28,30,32$.
Thus, the difference \emph{increases} by a constant amount of $10$.
Indeed, the diagonal lines are a distance of five apart.  Hence the
graph of $S(15,e)-C(15,e)$ for $36 \leq e \leq 39$ is linear with a
slope of $10$.  But for $e=40$, the adjacency matrix for the \qs
graph has an additional dot on the diagonal corresponding to $14$,
whereas the adjacency matrix for the \qc graph has an additional dot
on the diagonal corresponding to $24$.  So $S(15,40)-C(15,40)$
\emph{decreases} by $10$.  The decrease of $10$ continues until the
adjacency matrix for the \qc graph contains a complete column at
$e=45$.  Then the next matrix for $e=46$ has an additional dot in
the first row and next column and the slope changes again.

\section{Proof of Lemma \ref{lem:diagonal}} \label{sec:DiagonalLemma}

Returning for a moment to the threshold graph $\thr(\pi)$ from Figure \ref{fig:example643}, which  corresponds to the distinct partition $\pi=(6,4,3)$, we see  the graph complement  shown with the white dots.  Counting white dots in the rows from bottom to top and from the left to the diagonal, we have 7,5,2,1. These same numbers appear in columns reading from right to left and then top to the diagonal.  So if $\thr(\pi)$ is the threshold graph associated with $\pi$, then the set-wise complement of $\pi$ ($\pi^c$) in the set $\{1,2, \ldots, v-1\}$ corresponds to the threshold graph $\thr(\pi)^c$---the complement of $\thr(\pi)$.  That is,
\[
\thr(\pi^c)=\thr(\pi)^c.
\]
The diagonal sequence allows us to evaluate the sum of squares of the degree sequence of a threshold graph.  Each black dot contributes a certain amount to the sum of squares.  The amount depends on the location of the black dot in the adjacency matrix.  In fact all of the dots on a particular diagonal line contribute the same amount to the sum of squares.    For $v=8$, the value of a black dot in position $(i,j)$ is given by the entry in the following matrix:
\[
\left[ \begin{array}{cccccccc}
+&1&3&5&7&9&11&13\\
1&+&3&5&7&9&11&13\\
1&3&+&5&7&9&11&13\\
1&3&5&+&7&9&11&13\\
1&3&5&7&+&9&11&13\\
1&3&5&7&9&+&11&13\\
1&3&5&7&9&11&+&13\\
1&3&5&7&9&11&13&+\\
\end{array}
\right]
\]
This follows from the fact that a sum of consecutive odd integers is a square.  So to get the sum of squares $P_2(\thr(\pi))$ of the degrees of the threshold graph associated with the distinct partition $\pi$, sum the values in the numerical matrix above that occur in the positions with black dots.  Of course, an adjacency matrix is symmetric. So if we use only the black dots in the upper triangular part, then we must replace the $(i,j)$-entry in the upper-triangular part of the matrix above with the sum of the $(i,j)$- and the $(j,i)$-entry, which gives the following matrix:

\begin{equation} \label{eqn:evenmatrix}
E=
\left[
\begin{array}{cccccccc}
+&2&4&6&8&10&12&14\\
&+&6&8&10&12&14&16\\
&&+&10&12&14&16&18\\
&&&+&14&16&18&20\\
&&&&+&18&20&22\\
&&&&&+&22&24\\
&&&&&&+&26\\
&&&&&&&+
\end{array}
\right].
\end{equation}
Thus, $P_2(\thr(\pi))=2(1,2,3, \ldots ) \cdot \delta(\pi)$.  Lemma \ref{lem:diagonal} is proved.

\section{Proofs of Theorems \ref{thm:main1} and \ref{thm:main2}} \label{sec:proofsmain12}

Theorem \ref{thm:main1} is an immediate consequence of Theorem \ref{thm:main2} (and Lemmas \ref{lem:OptAreThres} and \ref{lem:diagonal}).  And Theorem \ref{thm:main2} can be proved using the following central lemma:

\begin{lemma} \label{lem:vminoneqc}
   Let $\pi=(v-1,c,c-1, \ldots , \widehat{j}, \ldots, 2,1)$ be an optimal partition in $\ds(v,e)$, where $e-(v-1)=1+2+ \cdots + c -j \geq 4$ and $1\leq j \leq c<v-2$.  Then $j=c$ and $2c \geq v-1$ so that
   \[
   \pi=(v-1, c-1, c-2, \ldots ,2,1).
   \]
\end{lemma}

We defer the proof of Lemma \ref{lem:vminoneqc} until Section \ref{sec:proofvminoneqc} and proceed now with the proof of Theorem \ref{thm:main2}.
The proof of Theorem \ref{thm:main2} is an induction on $v$.  Let $\pi$ be an optimal partition in $\ds(v,e)$,  then $\pi^c$ is optimal in $\ds(v,e^\prime)$.  One of the partitions, $\pi, \pi^c$ contains the part $v-1$.  We may assume without loss of generality that $\pi=( v-1:\mu)$, where $\mu$ is a partition in $\ds(v-1,e-(v-1))$.    The cases where $\mu$ is a decreasing partition of $0,1,2,$ and $3$ will be considered later.  For now we shall assume that $e-(v-1) \geq 4$.

Since $\pi$ is optimal, it follows that $\mu$ is optimal and hence by the induction hypothesis, $\mu$ is one of the following partitions in $\ds(v-1,e-(v-1))$:

\begin{description}
\item[1.1a] $\mu_{1.1}=(v-2,\ldots ,k^\prime+1,j^\prime)$, the \qs partition for $e-(v-1)$,
\item[1.2a]  $\mu_{1.2}=(v-2,\ldots ,\widehat{2k^\prime-j^\prime-1},\ldots
,k^\prime-1)$, if $k^\prime+1 \leq 2k^\prime-j^\prime-1 \leq v-2$,
\item[1.3a]  $\mu_{1.3}=(v-2,\ldots ,k^\prime+1,2,1)$, if $j^\prime=3$,
\item [2.1a]  $\mu_{2.1}=(k_1,k_1-1, \ldots, \widehat{j_1} , \ldots,2,1)$,
the \qc partition for $e-(v-1)$,
\item [2.2a]  $\mu_{2.2}=(2k_1-j_1-1,k_1-2,k_1-3, \ldots 2,1)$, if $k_1+1 \leq 2k_1-j_1-1 \leq v-2$,
\item [2.3a]  $\mu_{2.3}=(k_1,k_1-1, \ldots, 3)$, if $j_1=3$,
\end{description}
where
\[
e-(v-1) = 1+2+ \cdots + k_1-j_1\geq 4, \text{ with } 1 \leq j_1 \leq k_1.
\]
In symbols, $\pi=(v-1,\mu_{i.j})$, for one of the partitions $\mu_{i.j}$ above.  For each partition, $\mu_{i.j}$, we will show that $(v-1,\mu_{i.j})=\pi_{s.t}$ for one of the six partitions, $\pi_{s.t}$, in the statement of Theorem \ref{thm:main2}.

The first three cases are obvious:
\begin{eqnarray*}
(v-1,\mu_{1.1}) &=& \pi_{1.1}\\
(v-1,\mu_{1.2}) &=& \pi_{1.2}\\
(v-1,\mu_{1.3}) &=& \pi_{1.3}.
\end{eqnarray*}

Next suppose that $\mu=\mu_{2.1}, \mu_{2.2}, \text{ or }\mu_{2.3}$.
The partitions $\mu_{2.2}$ and $\mu_{2.3}$ do not exist unless
certain conditions on $k_1,j_1$, and $v$ are met.  And whenever
those conditions are met, the partition $\mu_{2.1}$  is also
optimal.  Thus $\pi_1=(v-1,\mu_{2.1})$ is optimal. Also, since
$e-(v-1)\geq 4$, then $k_1\geq 3$. There are two cases: $k_1=v-2,
k_1\leq v-3$. If $k_1=v-2$, then $\mu_{2.2}$ does not exist and
\[
(v-1,\mu) = \left\{ \begin{array}{l@{\text{  if }}l}
    \pi_{2.1}, & \mu=\mu_{2.1} \\
    \pi_{1.1}, & \mu=\mu_{2.3}.
    \end{array} \right.
\]

If $k_1\leq v-3$, then
by Lemma \ref{lem:vminoneqc},  $\pi_1=(v-1, k_1-1, \ldots,2,1)$,
with $j_1=k_1$ and $2k_1 \geq v-1$.
We will show  that $k=k_1+1$ and $v-1=2k-j-1$.
The above inequalities imply that
\begin{eqnarray*}
\binom{k_1+1}{2}=1+ 2+ \cdots + k_1 & \leq & e\\
&=& \binom{k_1+1}{2} - k_1 + (v-1)\\
&<&\binom{k_1+1}{2} + (k_1+1)=\binom{k_1+2}{2}.
\end{eqnarray*}
But $k$ is the unique integer satisfying $\tbinom{k}{2} \leq e <
\tbinom{k+1}{2}$. Thus $k=k_1+1$.

It follows that
\begin{eqnarray*}
e= (v-1)+1+2+ \cdots +(k-2)=\binom{k+1}{2}-j,
\end{eqnarray*}
and so $2k-j=v$.

We now consider the cases 2.1a, 2.2a, and 2.3a individually.  Actually, $\mu_{2.2}$ does not exist since $k_1=j_1$.  If  $\mu=\mu_{2.3}$,  then $\mu=(3)$ since $k_1=j_1=3$.  This contradicts the assumption that $\mu$ is a partition of an integer greater than 3.  Therefore
\[
\mu=\mu_{2.1}=( k_1,k_1-1, \dots , \widehat{j_1}, \ldots ,2,1)=(k-2, k-3, \ldots 2,1),
\]
since $k_1=j_1$ and $k=k_1+1$.  Now since $2k-j-1=v-1$ we have
\[
\pi =(2k-j-1,k-2,k-3,\ldots 2,1)=\left\{
\begin{array}{l}
\pi _{2.1}\text{ if }e=\tbinom{v}{2}\text{ or }e=\tbinom{v}{2}-(v-2) \\
\pi _{2.2}\text{ otherwise.}%
\end{array}%
\right. .
\]

Finally, if $\mu$ is a decreasing partition  of $0,1,2,$ or $3$,
then either $\pi=(v-1,2,1)=\pi_{1.3}$, or $\pi=(v-1)=\pi_{1.1}$, or
$\pi=(v-1,j^\prime)=\pi_{1.1}$ for some $1\leq j^\prime \leq 3$.

Now, we prove that $\pi _{1.2}$ and $\pi _{1.3}$ (if
they exist) have the same diagonal sequence as $\pi _{1.1}$ (which always exists). This in turn implies
(by using the duality argument mentioned in Section \ref{sec:DiagonalLemma}) that
$\pi _{2.2}$ and $\pi _{2.3}$ also have the same diagonal
sequence as $\pi_{2.1}$ (which always exists). We use the following observation. If we index the rows and
columns of the adjacency matrix $\adj(\pi )$ starting at zero instead
of one, then two positions $(i,j)$ and $(i^{\prime },j^{\prime })$
are in the same diagonal if and only if the sum of their entries are
equal, that is, $i+j=i^{\prime }+j^{\prime }$. If $\pi_{1.2}$ exists
then $2k^\prime-j^\prime \leq v$. Applying the previous argument to
$\pi _{1.1}$ and $\pi _{1.2}$, we observe that the top row of the
following lists shows the positions where there is a black dot in
$\adj(\pi _{1.1})$ but not in $\adj(\pi _{1.2})$
and the bottom row shows the positions where there is a black dot in $%
\adj(\pi _{1.2})$ but not in $\adj(\pi _{1.1})$.%
\[
\begin{tabular}{ccccc}
    $(v-k^{\prime }-2,v-1)$ & $\ldots$ & $(v-k^{\prime }-t,v-1)$ & $\ldots$ & $(v-k^{\prime }-(k^{\prime}-j^{\prime }),v-1)$
  \\$(v-1-k^{\prime },v-2)$ & $\ldots$ & $(v-1-k^{\prime },v-t)$ & $\ldots$ & $(v-1-k^{\prime },v-(k^{\prime}-j^{\prime })).$
\end{tabular}
\]
Each position in the top row is in the same diagonal as the corresponding
position in the second row. Thus the number of positions per diagonal is the
same in $\pi _{1.1}$ as in $\pi _{1.2}$. That is, $\delta \left( \pi
_{1.1}\right) =\delta \left( \pi _{1.2}\right) $.

Similarly, if $\pi_{1.3}$ exists then $k^\prime \geq j^\prime=3$. To
show that $\delta \left( \pi _{1.1}\right) =\delta \left( \pi
_{1.3}\right) $ note that the only position where there is a black
dot in $ \adj(\pi _{1.1})$ but not in $\adj(\pi _{1.3})$ is
$(v-1-k^{\prime },v-1-k^{\prime }+3)$, and the only position where
there is a black dot in $ \adj(\pi _{1.3})$ but not in $\adj(\pi
_{1.1})$ is $(v-k^{\prime },v-1-k^{\prime }+2)$. Since these
positions are in the same diagonal then $ \delta \left( \pi
_{1.1}\right) =\delta \left( \pi _{1.3}\right) $.

Theorem \ref{thm:main2} is proved.

\section{Proof of Lemma \ref{lem:vminoneqc}} \label{sec:proofvminoneqc}

There is a variation of the formula for $P_2(\thr(\pi))$ in Lemma \ref{lem:diagonal} that is useful in the proof of Lemma \ref{lem:vminoneqc}.   We have seen that each black dot in the adjacency matrix for a threshold graph contributes a summand, depending on the location of the black dot in the matrix $E$ in (\ref{eqn:evenmatrix}).  For example,  if $\pi=(3,1)$, then the part of $(1/2)E$  that corresponds to the black dots in the adjacency matrix  $\adj(\pi)$ for $\pi$ is
\[
\adj((3,1))=\begin{bmatrix}

+&\D&\D&\D\\
  &+&\D&\ci\\
&&+&\ci\\
 &&&+
 \end{bmatrix}, \qquad
\begin{bmatrix}

+&1&2&3\\
  &+&3&\\
&&+&\\
 &&&+
 \end{bmatrix}
\]
Thus $P_2(\thr(\pi))=2(1+2+3+3)=18$.  Now if we index the rows and columns of the adjacency matrix starting with zero instead of one, then the integer appearing in the matrix $(1/2)E$ at entry $(i,j)$ is just $i+j$.  So we can compute $P_2(\thr(\pi))$ by adding all of the positions $(i,j)$ corresponding to the positions of black dots in the upper-triangular part of the adjacency matrix of $\thr(\pi)$.  What are the positions of the black dots in the adjacency matrix for the threshold graph corresponding to a partition $\pi=(a_0,a_1, \ldots, a_p)$?  The positions corresponding to $a_0$ are
\[
(0,1), (0,2), \ldots, (0,a_0)
\]
and the positions corresponding to $a_1$ are
\[
(1,2), (1,3) \ldots, (1,1+a_1).
\]
In general, the positions corresponding to $a_t$ in $\pi$ are
\[
(t,t+1), (t,t+2), \ldots, (t,t+a_t).
\]
We use these facts  in the proof of  Lemma \ref{lem:vminoneqc}.

Let $\mu=(c,c-1, \ldots, \widehat{j}, \ldots ,2,1)$ be the \qc partition in $\ds(v,e-(v-1))$, where $1 \leq j \leq c <v-2$ and $1+2+ \ldots +c-j\geq4$.
We deal with the cases $j=1$, $j=c$, and $2 \leq j \leq c-1$ separately.  Specifically, we show that  if $\pi = (v-1:\mu)$ is  optimal, then $j=c$ and
\begin{equation} \label{eqn:optform2}
\pi=(v-1,c-1, \ldots,2,1),
\end{equation}
with  $2c \geq v-1$.

Arguments for the cases are given below.

\subsection{$j=1$ : $\mu=(c, c-1, \ldots, 3,2)$}
Since $2+3+\ldots+c \geq 4$ then $c\geq3$. We show that $\pi=(v-1:\mu)$ is not optimal.
In this case, the adjacency matrix for $\pi$ has the following  form:
\setlength{\arraycolsep}{3pt}
{
\renewcommand{\arraystretch}{1.0}
\[
\begin{array}{c|cccccccccc}
    &  0&   1&    2& \cdots &   &c&&&\cdots&  v-1\\ \hline
0  &  +&  \D&  \D&\cdots& & \D&  \D&  \D&\cdots&  \D \\
1  &    &   +&  \D&  \cdots   &    &  \D&  \D&  \ci&   \cdots      &  \ci\\
2 & & & +\\
\vdots&&&&\ddots \\
c-1    &    &     &   &   &     +    &  \bk &  \D& \ci & \cdots &\ci\\
c  &    &     &     &    &     &   +&  \ci& \ci & \cdots & \ci\\
c+1&  &     &     &     &    &     & + & \ci & \cdots & \ci\\
\vdots &&&& &&&&\ddots&&\vdots\\
&&&&&&&&&&\ci\\
v-1 & &&&&&&&&&+
\end{array}
\]
}
%\newpage
%\renewcommand{\arraystretch}{1.3}
%\[
%\begin{array}{c|ccccccccccccc}
%    &  \sd{0}&   \sd{1}&    \sd{2}& \cdots & \sd{c-1}  &\sd{c}&\sd{c+1}&&\cdots& \sd{2c-1}&\sd{2c}&\cdots& \sd{v-1}\\ \hline
%0  &  +&  \D&  \D&\cdots& & \D&  \D&  \D&\cdots&  \D &\D&\cdots&\D\\
%1  &    &   +&  \D&  \cdots   &    &  \D&  \D&  \ci&   \cdots      &  \ci&\ci&\cdots&\ci\\
%2 & & & +\\
%\vdots&&&&\ddots \\
%c-1    &    &     &   &   &     +    &  \bk &  \D& \ci & \cdots &\ci& \ci & \cdots &\ci\\
%c  &    &     &     &    &     &   +&  \ci& \ci & \cdots & \ci& \ci & \cdots &\ci\\
%c+1&  &     &     &     &    &     & + & \ci & \cdots & \ci& \ci & \cdots &\ci\\
%\\
%\\
%\vdots &&&& &&&&&&\ddots&&&\vdots\\
%\\
%&&&&&&&&&&&&+&\ci\\
%v-1 & &&&&&&&&&&&&+
%\end{array}
%\]

%\[
%\begin{array}{c|ccccccccccccc}
%    &  \sd{0}&   \sd{1}&    \sd{2}& \cdots & \sd{c-1}  &\sd{c}&\sd{c+1}&&\cdots& \sd{2c-1}&\sd{2c}&\cdots& \sd{v-1}\\ \hline
%0  &  +&  \D&  \D&\cdots& & \D&  \D&  \D&\cdots&  \D &\D&\cdots&\D\\
%1  &    &   +&  \D&  \cdots   &    &  \D&  \D&  \ci&   \cdots      &  \ci&\ci&\cdots&\ci\\
%2 & & & +\\
%\vdots&&&&\ddots \\
%c-1    &    &     &   &   &     +    &  \bk &  \D& \ci & \cdots &\ci& \ci & \cdots &\ci\\
%c  &    &     &     &    &     &   +&  \ci& \ci & \cdots & \ci& \ci & \cdots &\ci\\
%c+1&  &     &     &     &    &     & + & \ci & \cdots & \ci& \ci & \cdots &\ci\\
%\\
%\\
%\vdots &&&& &&&&&&\ddots&&&\vdots\\
%\\
%&&&&&&&&&&&&+&\ci\\
%v-1 & &&&&&&&&&&&&+
%\end{array}
%\]

\subsubsection{$2c \leq v-1$}
 Let
\[
\pi^\prime=(v-1,2c-1,c-2,c-3, \ldots, 3,2).
\]
The parts of $\pi^\prime$ are distinct and decreasing since $2c \leq v-1$.  Thus $\pi^\prime \in \ds(v,e)$.

The adjacency matrices  $\adj(\pi)$ and  $\adj(\pi^\prime)$ each
have $e$  black  dots,  many of which appear in the same positions.
But there are differences.  Using the fact that $c-1 \geq 2$, the
first row of the following list shows the positions in which a black
dot appears in $\adj(\pi)$ but not in $\adj(\pi^\prime)$.  And the
second row shows the positions in which a black dot appears in
$\adj(\pi^\prime)$ but not in $\adj(\pi)$:
\[
\begin{array}{cccccc}
(2,c+1)&(3,c+1)& \cdots & (c-1,c+1)&&(c-1,c)\\
(1,c+2)&(1,c+3)& \cdots & (1,2c-1)&&(1,2c)
\end{array}
\]
For each of the positions in the list, except the last ones, the sum of the coordinates for the positions is the same in the first row as it is in the second row.  But the coordinates of the last pair in the first row sum to $2c-1$ whereas the coordinates of the last pair in the second row sum to $2c+1$.  It follows that $P_2(\pi^\prime)=P_2(\pi)+4$.  Thus, $\pi$ is not optimal.

%
%\framebox{\parbox{6in}{{\tt REMOVE LATER

%Move $c-2$ black dots from column $c+1$ to row 1 as follows:
%\[
%(t, c+1) \rightarrow (1,c+t), \qquad t=2,\ldots,c-1.
%\]
%The sum of squares is unchanged and
%the right-most dot now in row 1 is $(1, 2c-1)$.  Since $2c \leq v-1$, there is room for at least one more dot in row 1 at position $(1,2c)$.
%Move $(c-1,c)$ to that position. This increases the sum of squares.

%If $c=2$, then $\pi$ has only two parts and is already \qs.}}}

\subsubsection{$2c > v-1$}
%\framebox{\parbox{6in}{\tt REMOVE LATER
%In this case move the dot in position $(0,v-1)$ to position $(c,c+1)$.  This increases the sum of squares.}}

Let $\pi^\prime=(v-2,c,c-1, \ldots,3,2,1)$.  Since $c<v-2$, the partition $\pi^\prime$ is in $\ds(v,e)$.  The positions of the black dots in the adjacency matrices $\adj(\pi)$ and $\adj(\pi^\prime)$ are the same but with only two exceptions.  There is a black dot in position $(0,v-1)$ in $\pi$ but not in $\pi^\prime$, and there is a black dot in position $(c,c+1)$ in $\pi^\prime$ but not in $\pi$.  Since $c + (c+1) > 0 +(v-1)$, $\pi$ is not optimal.

\subsection{$j=c: \mu=(c-1, \ldots,2, 1)$}
Since $1+2+ \cdots +(c-1) \geq 4$, then $c\geq 4$. We will show that
if $2c \geq v-1$, then $\pi$ has the same diagonal sequence as the
\qc partition. And if $2c<v-1$, then $\pi$ is not optimal.

The adjacency matrix for $\pi$ is of the following form:

%\framebox{\parbox{6in}{\tt REMOVE: corresponds to a dot pattern with corners in positions $(0,v-1)$, $(b,c+1)$, and $(c-1,c)$, where $b \leq c \leq v-3$.  }}

\renewcommand{\arraystretch}{1.3}
\[
\begin{array}{c|cccccccc}
    &  0&   1&    2&\cdots &   c&&\cdots&  v-1\\ \hline
0  &  +&  \D&  \D&\cdots&  \D&  \D&  \cdots&  \D\\
1  &    &   +&  \D&     &    \D&  \ci&          &  \ci\\
\vdots& &    & \ddots&  &\\
    &    &     &   &      +    &  \D&  \ci&  \cdots &\ci\\
c  &    &     &     &        &   +&  \ci&  \cdots & \ci\\
&  &     &     &     &         & + & \cdots & \ci\\
&&&&&&&\ddots\\
v-1&&&&&&&&+
\end{array}
\]
\subsubsection{$2c \geq v-1$}
The \qc partition in $\G(v,e)$ is $\pi^\prime=(c+1,c, \ldots , \widehat{k}, \ldots, 2,1)$, where $k=2c-v+2$.  To see this, notice that
\[
1+2+ \cdots + c+(c+1) -k=1+2+ \cdots + (c-1) + (v-1)
\]
for $k=2c-v+2$.  Since $2c \geq v-1$ and $c<v-2$, then $1 \leq k<c$ and $\pi^\prime \in \ds(v,e)$.

To see that $\pi$ and $\pi^\prime$ have the same diagonal sequence,
we again make a list of the positions in which there is a black dot
in $\adj(\pi)$ but not in $\adj(\pi^\prime)$ (the top row below),
and the positions in which there is a black dot in
$\adj(\pi^\prime)$ but not in $\adj(\pi)$ (the bottom row below):

\[
\begin{array}{cccccc}
(0,c+2)&(0,c+3)& \cdots &(0,c+t+1)& \cdots &(0,v-1)\\
(1,c+1)&(2,c+1)& \cdots&(t,c+1)& \cdots &(v-c-2,c+1).
\end{array}
\]
Each position in the top row is in the same diagonal as the
corresponding position in the bottom row, that is,
$0+(c+t+1)=t+(c+1)$.  Thus the diagonal sequences
$\delta(\pi)=\delta(\pi^\prime)$.

\subsubsection{$2c < v-1$}
In this case, let $\pi^\prime=(v-1,2c-2,c-3, \ldots, 3,2)$.  And
since $2c-2 \leq v-3$, the parts of $\pi^\prime$ are distinct and
decreasing.  That is, $\pi^\prime \in \ds(v,e)$.

Using the fact that $c-2 \geq 2$, we again list the positions in
which there is a black dot in $\adj(\pi)$ but not in
$\adj(\pi^\prime)$ (the top row below), and the positions in which
the is a black dot in $\adj(\pi^\prime)$ but not in $\adj(\pi)$:
\[
\begin{array}{ccccc}
(2,c)&(3,c)& \cdots&(c-1,c)&(c-2,c-1)\\
(1,c+1)&(1,c+2)&\ldots&(1,2c-2)&(1,2c-1).
\end{array}
\]
All of the positions but the last in the top row are on the same diagonal as the corresponding position in the bottom row: $t+c=1+(c-1+t)$.  But in the last positions we have $(c-2)+(c-1)=2c-3$ and $1+(2c-1)=2c$.  Thus $P_2(\pi^\prime)=P_2(\pi)+6$ and so $\pi$ is not optimal.

\subsection{$1<j<c: \mu=(c,c-1, \ldots, \widehat{j}, \ldots, 2,1)$}

We will show that $\pi=(v-1,c,c-1, \ldots, \widehat{j}, \ldots, 2,1)$ is not optimal.  The adjacency matrix for $\pi$ has the following form:
\renewcommand{\arraystretch}{1}
\[
\begin{array}{l|cccccccccc}
    &  0&   1&    2& \qquad \qquad \cdots \qquad \qquad &   \sd{c-1}&\sd{c}&\sd{c+1}&\sd{c+2}&\cdots&  \sd{v-1}\\ \hline
0  &  +&  \D&  \D&\cdots& & \D&  \D&  \D&\cdots&  \D\\
1  &    &   +&  \D&     &    &  \D&  \D&  \ci&         &  \ci\\
\vdots \\
c-j  &    &     &     &     &    &  \D&  \D&  \ci& \cdots&\ci\\
c-j+1&    &     &    &  \ddots   &   &  \D&  \ci&  \ci& \cdots&\ci\\
\vdots& &     &    &   &\\
c-1    &    &     &   &   &     +    &  \D&  \ci& \ci & \cdots &\ci\\
c  &    &     &     &    &     &   +&  \ci& \ci & \cdots & \ci\\
c+1&  &     &     &     &    &     & + & \ci & \cdots & \ci\\
\vdots&&&&&&&&+\\
&&&&&&&&&\ddots\\
v-1&&&&&&&&&&+
\end{array}
\]
There are two cases.

\subsubsection{$2c > v-1$}

Let $\pi^\prime=(v-r,c,c-1,\ldots ,\widehat{j+1-r},\ldots ,2,1)$, where $r=\min (v-1-c,j)$. Then $r>1$ because $j>1$ and $c<v-2$. We show that $\pi^\prime \in \ds(v,e)$ and $P_2(\pi^\prime)
>P_2(\pi)$.

In order for $\pi^\prime$ to be in $\ds(v,e)$, the sum of the parts in $\pi^\prime$ must equal the sum of the parts in $\pi$:
\[
1+2+\ldots +c+(v-r) -(j+1-r)=1+2+\ldots +c+(v-1)-j.
\]
And the parts of $\pi^\prime$ must be distinct and decreasing:
\[
v-r>c>j+1-r>1.
\]
The first inequality holds because $v-1-c\geq r$. The last two
inequalities hold because $c>j>r>1$. Thus $\pi^\prime \in \ds(v,e)$.

The top row below lists the positions where there is a black dot in $\adj(\pi)$ but not in $\adj(\pi^\prime)$; the bottom row lists the positions where there is a black dot in $\adj(\pi^\prime)$ but not in $\adj(\pi)$:
\setlength{\arraycolsep}{7pt}
\[
\begin{array}{lclcl}
(0,v-1)&\cdots&(0,v-t)&\cdots&(0,v-r+1)\\
(c-j+r-1,c+1)&\cdots&(c-j+r-t,c+1)& \cdots &(c-j+1,c+1).
\end{array}
\]
Since $r>1$ then the lists above are non-empty. Thus, to ensure that
$P_2(\pi^\prime)
> P_2(\pi)$, it is sufficient to show that for each $1\leq t\leq
r-1$, position $(0,v-t)$ is in a diagonal to the left of position
$(c-j+r-t,c+1)$. That is,
\[
0< [(c-j+r+1-t)+(c+1)]-[0+(v-t)]=2c+r-v-j,
\]
or equivalently,
\[
v-2c+j-1 \leq r=\min (v-1-c,j).
\]
The inequality $v-2c+j\leq v-1-c$ holds because $j<c$, and $v-2c+j\leq j$ holds because $v-1<2c$. It follows that $\pi$ is not an optimal partition.

\subsubsection{$2c \leq v-1$}

Again we show that $\pi=(v-1,c,c-1, \cdots, \widehat{j}, \cdots, 2,1)$ is not optimal.  Let
\[
\pi^\prime = (v-1,2c-2,c-2, \cdots, \widehat{j-1}, \cdots, 2,1).
\]
 The sum of the parts in $\pi$ equals the sum of the parts in $\pi^\prime$.  And the partition  $\pi^\prime$ is  decreasing:
\[
1 \leq j-1 \leq  c-2< 2c-2 <v-1.
\]
The first three inequalities follow from the assumption that $1<j <c$.  And the fourth inequality holds because $2c \leq v-1$.  So $\pi^\prime \in \ds(v,e)$.

The adjacency matrices $\adj(\pi)$ and $\adj(\pi^\prime)$ differ as follows.  The top rows of the following two lists contain the positions where there is a black dot in $\adj(\pi)$ but not in $\adj(\pi^\prime)$; the bottom row lists the positions where there is a black dot in $\adj(\pi^\prime)$ but not in $\adj(\pi)$.
\setlength{\arraycolsep}{7pt}
\[
\begin{array}{llclcl}
\text{List 1}&(2,c+1)& \cdots& (t,c+1) & \cdots &(c-j,c+1)\\
&(1,c+2)& \cdots & (1,c+t)& \cdots & (1,2c-j)\\
\\
\text{List 2}&(c-j+1,c)& \cdots& (c-j+t,c) & \cdots &(c-1,c)\\
&(1,2c-j+1)& \cdots & (1,2c-j+t)& \cdots & (1,2c-1).
\end{array}
\]
Each position, $(t,c+1)$  ($t=2, \ldots, c-j$), in the top row in List 1 is in the same diagonal as the corresponding position, $(1,c+t)$, in the bottom row of List 1. Each position, $(c-j+t,c)$ ($t=1, \ldots, j-1$), in the top row of List 2 is in a diagonal to the left of the corresponding position, $(1,2c-j+t)$ in the bottom row of List 2.  Indeed, $(c-j+t)+c = 2c-j+t < 2c-j+t+1 = 1+(2c-j+t)$.  And since $1<j$, List 2 is not empty.  It follows that $P_2(\pi^\prime)>P_2(\pi)$ and so $\pi$ is not a optimal partition.

The proof of Lemma \ref{lem:vminoneqc} is complete.

%%%%
\newpage
\section{Proof of Theorem \ref{thm:main3} and Corollaries \ref{cor:uniform} and \ref{cor:strictuniform}} \label{sec:main3}

The notation in this section changes a little from that used in Section \ref{sec:intro}.  In Section  \ref{sec:intro}, we write $e=\binom{k+1}{2}-j,$ with $1 \leq j \leq k$.  Here, we let $t=k-j$ so that
\begin{equation} \label{eqn:repre}
e=\binom{k}{2}+t,
\end{equation}
with $0 \leq t \leq k-1$.  Then Equation (\ref{eqn:Cve})  is equivalent to
\begin{equation} \label{eqn:Ckt}
C(v,e) = C(k,t) = (k-t)(k-1)^2+tk^2+t^2=k(k-1)^2+t^2+t(2k-1).
\end{equation}
Before proceeding, we should say that the abuse of notation in $C(v,e)=C(k,t)$ should not cause confusion as it will be clear which set of parameters $(v,e)$ {\it vs.} $(k,t)$ are being used.  Also notice that if we were to expand the range of $t$ to $0 \leq t \leq k$, that is allow $t=k$, then the representation of $e$ in Equation (\ref{eqn:repre}) is not unique:
\begin{eqnarray*}
e &=& \binom{k}{2}+k = \binom{k+1}{2}+0.
\end{eqnarray*}
But the value of $C(v,e)$ is the same in either case:
\[
C(k,k)=C(k+1,0)= (k+1)k^2.
\]
Thus we may take $0 \leq t \leq k$.

We begin the proofs now. At the beginning of Section
\ref{sec:QSvsQC}, we showed that $S(v,e)=C(v,e)$ for $e=0,1,2,3$. Also note that, when $m$ is an integer, $\DF(v,m)=0$. We
now compare $S(v,e)$ with $C(v,e)$ for $4 \leq e < m$.  The first
task is to show that $S(v,e) > C(v,e)$ for all but a few values of
$e$ that are close to $m$.  We start by finding upper and lower
bounds on $S(v,e)$ and $C(v,e)$.

%Write $e=\binom{k}{2}+t$, $0 \leq t \leq k$. Then
%$C(v,e)=C(k,t)=k(k-1)^2+t^2+(2k-1)t$.
Define
\begin{eqnarray*}
 U(e) & = & e(\sqrt{8e+1}-1)\mbox{ and} \\
 U(k,t) & = & \left(\binom{k}{2}+t\right)\left(\sqrt{(2k-1)^2+8t}-1\right).
\end{eqnarray*}

The first lemma shows that $U(e)$ is an upper bound for $C(v,e)$ and
leads to an upper bound for $S(v,e)$.  Arguments used here to obtain upper and lower bounds are similar to those in \cite{Ni}.
\begin{lemma} \label{lem:upperbS}
     For $e \geq 2$
     \begin{eqnarray*}
            C(v,e) & \leq & U(e) \mbox{ and } \\
            S(v,e) & \leq & U(e') + (v-1)(4e-v(v-1)).
     \end{eqnarray*}
\end{lemma}

It is clearly enough to prove the first inequality. The second
one is trivially obtained from Equation (\ref{eqn:SCconnection})
 linking the values of $S(v,e)$ and $C(v,e)$.

\begin{proof}
     We prove the inequality in each interval $\binom{k}{2} \leq e \leq \binom{k+1}{2}$. So fix $k \geq 2$ for now. We make yet another change of variables to get rid off the square root in the above expression of $U(k,t)$.

     Set $t(x) = (x^2-(2k-1)^2)/8$, for $2k-1 \leq x \leq 2k+1$. Then
      \begin{eqnarray*}
         \lefteqn{U(k,t(x)) - C(k,t(x))  = } \\ && \frac{1}{64} (x- (2 k -1)) ((2 k +1) - x) \left( x^2 + 4 (k-2) (k+x)-1 \right),
      \end{eqnarray*}
which is easily seen to be positive for all $k \geq 2$ and all
$2k-1 \leq x \leq 2k+1$.
\end{proof}

Now define
\begin{eqnarray*}
       L(e) & = & e(\sqrt{8e+1}-1.5) \mbox{ and} \\
       L(k,t) & = & \left(\binom{k}{2}+t\right)\left(\sqrt{(2k-1)^2+8t}-1.5\right).
\end{eqnarray*}

The next lemma shows that $L(e)$ is a lower bound for $C(v,e)$ and
leads to a lower bound for $S(v,e)$.
\begin{lemma} \label{lem:lowerbS}
     For $e \geq 3$
     \begin{eqnarray*}
            C(v,e) & \geq & L(e) \mbox{ and } \\
            S(v,e) & \geq & L(e') +(v-1)(4e-v(v-1)).
     \end{eqnarray*}
\end{lemma}

\begin{proof}
     As above, set $t(x) = (x^2-(2k-1)^2)/8$, $2k-1 \leq x \leq
     2k+1$, and $x(k,b)=2k+b$, $-1 \leq b \leq 1$. Then
      \begin{eqnarray*}
         \lefteqn{C(k,t(x(k,b))) - L(k,t(x(k,b))) =} \\
         &  & \frac{1}{64} b^2 (b+4 k-4)^2 +
         \frac{1}{32} (4 k-7) \left(b+\frac{2 (k+1)}{4 k-7}\right)^2+\frac{4 k (22 k-49)+13}{64 (4 k-7)}
      \end{eqnarray*}
This expression is easily seen to be positive for $k \geq 3$.
\end{proof}

We are now ready to prove that $S(v,e) > C(v,e)$ for $0 \leq e
\leq m$ for all but a few small values and some values close to
$m$.

\begin{lemma} \label{lemma:4<e<v}
      Assume $v \geq 5$. For $4 \leq e < v$ we have $C(v,e)<S(v,e)$.
\end{lemma}

\begin{proof}
  As we showed above in Lemma \ref{lem:upperbS}, $e(\sqrt{8e+1}-1)$ is an upper bound on $C(v,e)$ for all $1 \leq e \leq \binom{v}{2}$. Furthermore, it is easy to see that for $1 \leq e < v$ we have $S(v,e)=e^2+e$. In fact, the \qs graph is optimal for $1 \leq e <v$.  The rest is then straightforward. For $4 \leq e$, we have
  \begin{eqnarray*}
           0 < (e-3)(e-1) & = & (e+2)^2-(8e+1).
  \end{eqnarray*}
  Taking square roots and rearranging some terms proves the result.
\end{proof}

\begin{lemma} \label{lemma:0.55}
       Assume $v \geq 5$. For $v \leq e \leq m - 0.55v$ we have
       \begin{eqnarray*}
             S(v,e) & > & C(v,e).
       \end{eqnarray*}
\end{lemma}

\begin{proof}
Assume that $0 \leq e \leq m$. Let $e=m-d$ with $0 \leq d \leq m$.
By Lemmas \ref{lem:upperbS} and \ref{lem:lowerbS}, we have
\begin{eqnarray*}
    S(v,e) - C(v,e) & \geq & L(e') +(v-1)(4e-v(v-1)) - U(e) \\
                 & = & (m+d) \sqrt{8 (m+d)+1} -(m-d) \sqrt{8
                 (m-d)+1} \\
                 & & -\left( \left(4 (v-1)  + \frac{5}{2}\right)d
                 + \frac{m}{2} \right).
\end{eqnarray*}
We focus on the first two terms. Set
\begin{eqnarray}
    h(d) & = & (m+d) \sqrt{8 (d+m)+1} -(m-d) \sqrt{8 (m-d)+1}.
\end{eqnarray}
By considering a real variable $d$, it is easy to see that
$h^\prime(d)>0$, $h^{(2)}(0)=0$, and $h^{(3)}(d)<0$
%$\frac{\partial \, h(m,d)}{\partial d}
%> 0$, $\frac{\partial^2 \, h(m,d)}{\partial d^2}|_0=0$, and finally
%$\frac{\partial^3 \, h(m,d)}{\partial d^3} < 0$
on the interval in question. Thus $h(d)$ is concave down on $0 \leq
d \leq m$. We are comparing $h(d)$ with the line $(4(v-1)+5/2)d+m/2$
on the interval $[0.55 v,m-v]$. The concavity of $h(d)$ allows to
check only the end points. For $d=m-v$, we need to check
$$
 \frac{1}{2} v \left((v-3) \, \sqrt{4 v^2-12 v+1} -2 \sqrt{8 v+1} \right)
  >
 \frac{1}{4} v \left(4 v^2-21 v+7\right)
$$
It is messy and elementary to verify this inequality for $v\geq9$.

For $d=0.55v$ we need to check
$$
 \left(\frac{v^2}{4}+0.3 v\right) \sqrt{2 v^2+2.4 v+1}-
 \left(\frac{v^2}{4}-0.8 v\right) \sqrt{2 v^2-6.4 v+1}
 > v(2.325 v - 0.95).
$$
This inequality holds for $v \geq 29$. This time the calculations
are a bit messier, yet still elementary. For $4 < v \leq 28$, we
verify the result directly by calculation.
\end{proof}

In Section 1, we introduced the value $e_0= \binom{k_0}{2}$.

We now define
\begin{eqnarray*}
e_1 &=& \binom{k_0-1}{2}\\
f_1&=&\binom{v}{2}-\binom{k_0+1}{2}\\
f_2&=&\binom{v}{2}-\binom{k_0+2}{2}.
\end{eqnarray*}

The next lemma shows that those binomial coefficients and their
complements are all we need to consider.

\begin{lemma} \label{lemma:e1f2}
      $e_1,f_2 < m - 0.55v$.
\end{lemma}

As a consequence $S(v,e)>C(v,e)$ for all $4 \leq e \leq \max \{
e_1, f_2 \}$. We need a small result on the relationship between
$k_0$ and $v$ first. The upper bound will be used later in this
section.

\begin{lemma} \label{lem:kzerobound}
       $\frac{\sqrt{2}}{2} \left(v - \frac{1}{2} \right)- \frac{1}{2} < k_0 < \frac{\sqrt{2}}{2} v + \frac{1}{2}$.
\end{lemma}

\begin{proof}
    Since $\binom{k_0}{2} \leq m \leq \binom{k_0+1}{2}-\frac{1}{2}$, we have
    \[
      2k_0(k_0-1) \leq v^2-v \leq 2k_0(k_0+1)-2.
    \]
    Thus
    \[
       2(k_0-1/2)^2 \leq (v-1/2)^2 +1/4 \leq 2(k_0+1/2)^2-2.
    \]
    That is,
    \begin{eqnarray*}
       \frac{\sqrt{2}}{2} \sqrt{\left(v - \frac{1}{2} \right)^2 +\frac{9}{4}} -\frac{1}{2}
       \leq k_0 \leq
       \frac{\sqrt{2}}{2} \sqrt{\left(v - \frac{1}{2} \right)^2 + \frac{1}{4}}+\frac{1}{2}.
    \end{eqnarray*}
    The result follows using $(v-1/2)^2<(v-1/2)^2+9/4$ and $(v-1/2)^2+1/4<v^2$.
\end{proof}

\begin{proof}[Proof of Lemma \ref{lemma:e1f2}]
    Note that $e_1=e_0-(k_0-1) \leq m -(k_0-1)$ and $f_2=f_1-(k_0+1) <
    m-(k_0+1)<m-(k_0-1)$. Hence, it is enough to show that
    $0.55v < (k_0-1)$. This follows from the previous lemma for
    $v\geq12$. For $5 \leq v \leq 11$, we verify the statement by
    direct calculation.
\end{proof}

Next, we show that the difference function
\[
\DF(v,e)=S(v,e)-C(v,e)
\]
is piecewise linear on the intervals  induced by the binomial
coefficients $\binom{k}{2}$, $2 \leq k \leq v$, and their
complements $\binom{v}{2}-\binom{k}{2}$, $2 \leq k \leq v$. In
Section \ref{sec:piecelinearity}, we show a specific example.

\begin{lemma} \label{lem:linearity}
As a function of $e$, the function $\DF(v,e)$ is linear on the
interval $\max \{ \binom{k}{2},\binom{v}{2}-\binom{l+1}{2} \} \leq e
\leq \min \{ \binom{k+1}{2}, \binom{v}{2}-\binom{l}{2} \}$. The line
has slope
\begin{equation}\label{eqn:klslope}
-\frac{1}{4} \left( 1-(2k-3)^2-(2l-3)^2+(2v-5)^2 \right).
\end{equation}
\end{lemma}

\begin{proof}
If $e=\tbinom{k+1}{2}-j$ with $1 \leq j \leq k$, then it is easy to
see from Equation (\ref{eqn:Cve}) that
\begin{eqnarray*}
      C(v,e+1)-C(v,e) & = & 2e - 2\binom{k}{2} + 2k = 2e - k(k-3).
\end{eqnarray*}

Using Equations (\ref{eqn:SCconnection}) and (\ref{eqn:Ckt}), we find
that, if $e^\prime=\binom{l}{2}+c$, $1 \leq c \leq l$, then
\begin{eqnarray*}
       S(v,e+1) - S(v,e) & = & 2e+ 4(v-1)-2\binom{v}{2} - 2l + 2\binom{l}{2}+2.
\end{eqnarray*}
We now have
\begin{eqnarray*}
      \lefteqn{\left( S(v,e+1)-C(v,e+1) \right) - \left( S(v,e)-C(v,e)
      \right) } \\
      & = & k(k-3) + l(l-3) - (v-1)(v-4) + 2 \\
      & & \\
      & = & -\frac{1}{4} \left(1-(2k-3)^2-(2l-3)^2+(2v-5)^2
      \right).
\end{eqnarray*}

The conclusion follows.
\end{proof}

Since we already know that $\DF(v,e)>0$ for $4 \leq e \leq \max \{
e_1,f_2 \}$, and $\DF(v,e)=0$ for $e=0,1,2,3,$ or $m$, we can now focus on the interval $I_1=(\max \{ e_1,f_2
\},m)$. The only binomial coefficients or complements of binomial
coefficients that can fall into this interval are $e_0$ and $f_1$.

There are two possible arrangements we need to consider

\begin{enumerate}
  \item $e_1, f_2 < e_0  \leq f_1 < m$ and
  \item $f_1 < e_0 \leq m$.
\end{enumerate}

The next result deals with the first arrangement.

\begin{lemma}
   If $e_0 \leq f_1 < m$, then $q_0(v)>0$. Furthermore, $S(v,e) \geq C(v,e)$ for $0 \leq e \leq m$ with equality if
   and only if $e=0,1,2,3,$ or $m$; or $e=e_0$ and $(2v-3)^2-2(2k_0-1)^2 =-1,7$.
\end{lemma}

\begin{proof}
   $e_0 \leq f_1$ implies $e_0 \leq m-k_0/2$. By Lemma \ref{lem:kzerobound}, we conclude that for
   $v>12$,
   \begin{eqnarray*}
   4q_0(v) & = & 1-2(2k_0-3)^2+(2v-5)^2\\
   & = &16(m-e_0)-16(v-k_0)+8 \\
   & \geq & 24k_0 - 16 v + 8 \\
   & \geq & 24(\sqrt{2}/2 (v-1/2) -1/2) - 16 v+8 \\
   & = & (12\sqrt{2}-16)v -( 6\sqrt{2}+4) \\
   & > & 0.
   \end{eqnarray*}
   For smaller values, we verify that $q_0(v)>0$ by direct calculation.

If $e=f_1$ in Equation (\ref{eqn:Ckt}), and since $e_0 \leq f_1 <m$,
then $k=k_0$ and $t=f_1-\tbinom{k_0}{2}$. Using Equation
(\ref{eqn:SCconnection}), $\DF(v,f_1) =(m-f_1)q_0(v)>0$. Similarly,
since $f_2 < e_0 \leq f_1$, then for $e=e_0^\prime$ in Equation
(\ref{eqn:Ckt}), we have $k=k_0+1$ and
$t=e_0^\prime-\tbinom{k_0+1}{2}$. Again, using Equation
(\ref{eqn:SCconnection}),
\begin{eqnarray}
   \DF(v,e_0) & = & (v^2-3v-2k_0^2+2k_0+2)(v^2-3v-2k_0^2+2k_0)/4 \label{diffve0} \\
   & = &((2v-3)^2-2(2k_0-1)^2+1)((2v-3)^2-2(2k_0-1)^2-7)/64. \nonumber
\end{eqnarray}
Notice that $\DF(v,e_0)\geq 0$ since both factors in (\ref{diffve0})
are even and differ by 2. Equality occurs if and only if
$(2v-3)^2-2(2k_0-1)^2=-1$ or $7$. Finally, observe that
$\DF(v,e_1)>0$ and $\DF(v,f_2)>0$ by Lemmas \ref{lemma:0.55} and
\ref{lemma:e1f2}, and $e_1$ and $f_2$ are both less than $f_1$.
Hence $\DF(v,e) \geq 0$ for $e \in [\max \{ e_1,f_2 \},m]$ follows
from  the piecewise linearity of $\DF(v,e)$. The rest follows from
Lemma \ref{lemma:0.55}.
\end{proof}

%Now we deal with the case $f_1<e_0$. We note that this, together with $e_1<f_1$, implies
%\begin{equation} \label{eqn:Cond12}
%  -2-2k_0^2+2k_0+v^2-v > 0, \qquad 2k_0^2-v^2+v > 0.
%\end{equation}

Now we deal with the case $f_1<e_0$. There are three cases depending on the sign of $q_0(v)$. All these
cases require the following fact. If $f_1<e_0$, then for $e_0 \leq e
\leq m$ in Equation (\ref{eqn:Ckt}), $k=k_0$ and
$t=e-\tbinom{k_0}{2}$. Since $f_1<e\leq m$, for $e^\prime$ in
Equation (\ref{eqn:Ckt}), $k=k_0$ and $t=e^\prime-\tbinom{k_0}{2}$.
Thus, using Equation (\ref{eqn:SCconnection}),
\begin{equation} \label{eqn:diffnearm}
    \DF(v,e) =(m-e)q_0(v)
\end{equation}
whenever $f_1<e_0\leq e \leq m$. This automatically gives the sign of $\DF(v,e)$ near $m$. By the piecewise linearity of $\DF(v,e)$ given by Lemma \ref{lem:linearity}, the only thing remaining is to investigate the sign of $\DF(v,f_1)$.

\begin{lemma}
      Assume $f_1<e_0$ and $q_0(v)>0$. Then $S(v,e) \geq C(v,e)$ for $0 \leq e \leq m$, with equality if and only if $e=0,1,2,3,m$.
\end{lemma}

\begin{proof}
  First, note that $e_1 \leq f_1 < e_0 < m$, since $e_1>f_1$ occurs only if $m=e_0$ and thus $q_0(v)=2-4(v-k_0)<0$. For $e_0 \leq e < m$, by Equation (\ref{eqn:diffnearm}), $\DF(v,e) =(m-e)q_0(v)>0$. Furthermore, if $e=f_1$ in Equation (\ref{eqn:Ckt}), then $k=k_0-1$ and $t=f_1-\tbinom{k_0-1}{2}$. Thus, by Equation (\ref{eqn:SCconnection}),
  \begin{eqnarray*}
    \DF(v,f_1)& = & (-4 k_0^4+16 k_0^3+4 v^2 k_0^2-12 v k_0^2-8v^2 k_0+4 k_0-v^4+6v^3+v^2-6 v)/4,
  \end{eqnarray*}
  and
  \begin{eqnarray*}
       \DF(v,f_1)-\DF(v,e_0) & = &  (2k_0^2-v^2+v)(-2-2k_0^2+8k_0+v^2-5v)/2.
  \end{eqnarray*}
The first factor is positive because $f_1<e_0$. The second factor is positive for $v \geq 15$.  This follows from the fact that $v<\sqrt{2}k_0+(\sqrt{2}+1)/2$ by Lemma \ref{lem:kzerobound}, and $-2-2k_0^2+2k_0+v^2-v \geq 0$ because $e_1 \leq f_1$: For $v \geq
15$,
  \begin{eqnarray*}
     -2-2k_0^2+8k_0+v^2-5v&=&(-2-2k_0^2+2k_0+v^2-v)+2(3k_0-2v)\\
    & \geq & 2(3k_0-2v)\\
    & > &0,
  \end{eqnarray*}
    Since $\DF(v,e_0)>0$, then $\DF(v,f_1)>0$ for $v \geq 15$. The only case left to verify satisfying the conditions of this lemma is $v=14$. In this case, $f_1=36$ and $\DF(14,36)=30>0$.
\end{proof}

The previous two lemmas provide a proof of part 1 of Theorem
\ref{thm:main3}.

\begin{lemma}
  Assume $f_1<e_0$ and $q_0(v)=0$. Then $S(v,e) \geq C(v,e)$ for
  $0 \leq e \leq m$ with equality if and only if $e=0,1,2,3,e_0,e_0+1,\ldots,m$.
\end{lemma}

\begin{proof}
  For $e_0 \leq e \leq m$, by Equation (\ref{eqn:diffnearm}), $\DF(v,e) =
(m-e)q_0(v)=0$. As in the previous lemma, for $v\geq 15$
  \begin{eqnarray*}
     \DF(v,f_1)-\DF(v,e_0)= (2 k_0^2-v^2+v) (-2-2k_0^2+8k_0+v^2-5v)/2>0
  \end{eqnarray*}
  and thus $\DF(v,f_1)>0$. The only value of $v<15$ satisfying the conditions of this lemma is $v=6$ with $f_1=5$, and $\DF(6,5)=4>0$.
  \end{proof}

The previous lemma provides a proof for part 3 of Theorem \ref{thm:main3}.

\begin{lemma}
      Assume $f_1 < e_0 \leq m$ and $q_0(v)<0$. Then $S(v,e) \geq C(v,e)$
      for $0 \leq e \leq m-R_0$ and $S(v,e) \leq C(v,e)$ for $m-R_0 \leq e
      \leq m$ with equality if and only if $e=0,1,2,3,m-R_0,m$.
\end{lemma}

\begin{proof}
For $e_0 \leq e < m$, by Equation (\ref{eqn:diffnearm}), $\DF(v,e)=(m-e)q_0(v) < 0$. This time it is possible that $f_1 < e_1$. In this case, by Lemmas \ref{lemma:0.55} and \ref{lemma:e1f2}, we know that $\DF(v,f_1), \DF(v,e_1)>0$. Also, $m=e_0$ and $R_0=0$, implying $\DF(v,e_0)=0$ and $\DF(v,e)>0$ for all $e_1 \leq e<e_0=m-R_0=m$.

If $e_1 \leq f_1$, by Lemma \ref{lem:linearity}, $\DF(v,e)$ is linear as a function of $e$ on the interval $[f_1,e_0]$. Let $-q_1(v)$ be
the slope of this line. Since $e_1<f_1<e_0\leq m$, then $k=k_0$ and $l=k_0$ in Lemma \ref{lem:linearity}. Thus
$q_1(v)=(-1-2(2k_0-4)^2+(2v-5)^2)/4=q_0(v)+2k_0-4$ and $\DF(v,f_1)=(m-e_0)q_0(v)+(e_0-f_1)q_1(v)$. The line through the two points $(e_0,\DF(v,e_0))$ and $(f_1,\DF(v,f_1))$ crosses the $x$-axis at $m-R_0$. We now show that $f_1 < m-R_0 < e_0$, which in turn proves that $\DF(v,f_1)>0$.

We have
  \begin{eqnarray}\label{eqn:R_0 first}
     m-R_0 & = & e_0 + (m-e_0) \frac{q_0(v)}{q_1(v)} \\
            & = & m -(m-e_0) \frac{2k_0-4}{q_1(v)}\label{eqn:R_0 second}.
  \end{eqnarray}
Since $e_0 \leq m$ and $v>4$, then
\begin{eqnarray}\label{eqn:v_k0}
     k_0  \leq
     \frac{1}{2}+\sqrt{\binom{v}{2}+\frac{1}{4}}<2+\sqrt{\binom{v-2}{2}},
  \end{eqnarray}
  which is equivalent to $q_1(v)>0$. Thus $m-R_0<e_0$ by Equation (\ref{eqn:R_0 first}). To prove $f_1<m-R_0$, according to Equation
  (\ref{eqn:R_0 second}), we need to show
  \begin{eqnarray*}
     (m-e_0)\frac{2k_0-4}{q_1(v)} & < & \binom{k_0+1}{2}-m.
  \end{eqnarray*}
  After multiplying by $q_1(v)$, the last inequality becomes
  \begin{eqnarray*}
      \left(m-\binom{k_0+1}{2} +\frac{k_0}{2} \right)\left(2k_0-4\right) & < &
      \left(\binom{k_0+1}{2}-m\right)\left((v-2)(v-3)-2(k_0-2)^2\right),
  \end{eqnarray*}
  which is equivalent to
\begin{eqnarray*}
      \frac {k_0}{2}(2k_0-4) & < & \left(\binom{k_0+1}{2}-m\right)\left((v-2)(v-3)-2(k_0-2)(k_0-3)\right).
  \end{eqnarray*}
Since $f_1 < e_0$ we know that $k_0/2<\binom{k_0+1}{2}-m$. Also,
Inequality (\ref{eqn:v_k0}) is equivalent to
$2k_0-4<(v-2)(v-3)-2(k_0-2)(k_0-3)$. Multiplying these two
inequalities yields the result.
\end{proof}

The previous lemma provides a proof of part 2 of Theorem
\ref{thm:main3}.

The expression for $m-R_0$ is sometimes an integer. Those $v<1000$ for which $m-R_0$ is an integer are 14, 17, 21, 120, 224, 309, 376, 393, 428, 461, 529, 648, 697, and 801.

In the remaining part of this section, we prove Corollaries \ref{cor:uniform} and \ref{cor:strictuniform}.

\begin{lemma}
   Assume that $v>4$ and $q_0(v)<0$. Then $R_0 \leq \alpha v$ where $\alpha=1-\sqrt{2}/2$.
\end{lemma}

\begin{proof}
   We show that $R_0 \leq \alpha v$ for $v>4$. Recall that
   \begin{eqnarray*}
     R_0 & = & \frac{(m-e_0)(2k_0-4)}{q_1(v,k_0)}.
   \end{eqnarray*}
   Thus we need to show
   \begin{eqnarray*}
    \alpha v q_1(v,k_0) - (m-e_0)(2k_0-4) > 0.
   \end{eqnarray*}
   Define the function $h(x)=\alpha vq_1(v,x)-(m-\binom{x}{2})(2x-4)$. The interval for $x$ is limited by the condition that
   $q_0(v)<0$ which implies that
   \begin{eqnarray*}
      i_1:=\frac {\sqrt{2}}{2} v-\frac{5\sqrt{2}}{4}+\frac{3}{2} & < & k_0.
   \end{eqnarray*}
   Furthermore, since $e_0 \leq m$, we know that $i_2:=(\sqrt{2}/2)v+1/2 > k_0$. We show that $h(x)$ is increasing on $I:=[i_1,i_2]$.
   Note that, since $v>4$,
   \begin{eqnarray*}
       h''(x) & = & -6 -(4-2\sqrt{2})v + 6 x > 0
   \end{eqnarray*}
   for $x \in I$. Hence $h(x)$ is concave up on $I$.
   Furthermore
   \begin{eqnarray*}
       h'(i_1) & = & \left(3-2 \sqrt{2}\right)
       v^2+\left(-10+\frac{11}{2}
       \sqrt{2}\right) v-\frac{15}{4} \sqrt{2}+\frac{73}{8} > 0
   \end{eqnarray*}
   for $v \geq 11$, and hence
   \begin{eqnarray*}
     h(x) & > & h(i_1) \\
     & = & \frac{1}{32} \left(\left(-72+58 \sqrt{2}\right) v+23
     \left(6-5 \sqrt{2}\right)\right)>0
   \end{eqnarray*}
   for $v \geq 11$.
   The only values of $v$ greater than 4 and smaller than 11 for which $q_0(v)<0$ are
   $v=7,10$. The result is easily verified in those two cases.
\end{proof}

How good is the bound $R_0 \leq \alpha v$? Suppose there is a parameter $\beta$ such that $R_0\leq \beta v$ with $\beta<\alpha$.
Assume $q_0(v)=-2$. There are infinitely many values of $v$ for which this is true (see Section \ref{sec:Pell}). In all of those cases $k_0(v)=1/2\sqrt{(9+(2v-5)^2)/2}+3/2$. We have the following
\begin{eqnarray*}
(\beta v q_1(v) - (m-e_0)(2k_0-4))/v^2 & \rightarrow &  \sqrt{2}
\beta- \sqrt{2}+1 \geq 0
\end{eqnarray*}
as $v \rightarrow \infty$.   Thus $\beta \geq \alpha$ and hence $\alpha$ the greatest number for which the bound on $R_0$ holds.

Since $S(v,e) \geq C(v,e)$ for all $1 \leq e \leq m-R_0$, we proved Corollary \ref{cor:uniform}.

To prove Corollary \ref{cor:strictuniform}, we need to investigate the other non-trivial case of equality in Theorem \ref{thm:main3}.
It occurs when $e=e_0$ and $(2v-3)^2-2(2k_0-1)^2=-1,7$. Notice that this implies
\begin{eqnarray*}
m-e_0 & = & \frac{1}{16}\left( (2v-1)^2-2(2k_0-1)^2+1 \right) \\
& = & \frac{v}{2} \text{ or } \frac{v-1}{2}.
\end{eqnarray*}
There are infinitely many values of $v$ such that $(2v-3)^2-2(2k_0-1)^2=-1$, and infinitely many values of $v$ such
that $(2v-3)^2-2(2k_0-1)^2=7$ (see Section \ref{sec:Pell}). Thus the best we can say is that $S(v,e)>C(v,e)$ for all $4 \leq e < m - v/2$, and Corollary
\ref{cor:strictuniform} is proved.

\section{Proof of Corollary \ref{cor:density}}
Recall that for each $v$, $k_0(v)=k_0$  is the unique positive integer such that
\[
\binom{k_0}{2} \leq \frac{1}{2} \binom{v}{2} < \binom{k_0+1}{2}.
\]
It follows that
\begin{equation} \label{eqn:between}
-1 \leq (2v-1)^2-2(2k_0-1)^2, \text{ and } (2v-1)^2 -2(2k_0+1)^2
\leq -17.
\end{equation}
Let us restrict our attention to the parts of the hyperbolas
\[
H_{\text{low}}: (2v-1)^2-2(2k-1)^2=-1, \quad H_{\text{high}}:
(2v-1)^2-2(2k+1)^2=-17
\]
that occupy the first quadrant as shown in Figure \ref{fig:hyperb}.
Then each lattice point, $(v,k_0)$ is in the closed region bounded by $H_{\text{low}}$ below and $H_{\text{high}}$ above.
\begin{figure}[h]
\begin{center}
\epsfig{file=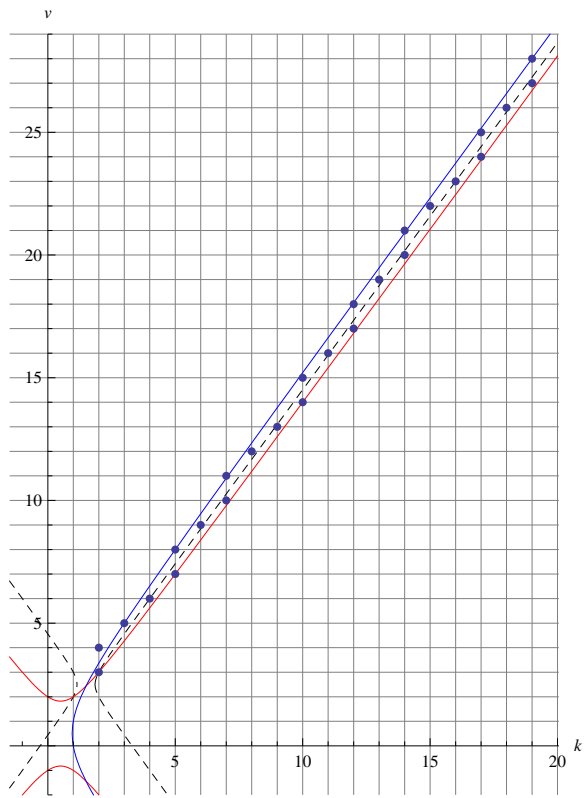,width=3in}
\caption{Hyperbolas
$(2v-1)^2-2(2k-1)^2=-1$, $(2v-1)^2-2(2k+1)^2 = -17$, $(2v-5)^2
-2(2k-3)^2=-1$} \label{fig:hyperb}
\end{center}
\end{figure}
Furthermore, the sign of the quadratic form $(2v-5)^2-2(2k-3)^2+1$ determines whether the \qs  graph is optimal in $\G(v,e)$ for all $0 \leq e \leq m$.  By Theorem \ref{thm:main3}, if $(2v-5)^2-2(2k-3)^2+1\geq 0$, then $S(v,e) \geq C(v,e)$ (and the \qs graph is optimal) for $0 \leq e \leq m$.  Thus, if the lattice point $(v,k)$ is between $H_{\text{high}}$ and the hyperbola
\[
H: (2v-5)^2-2(2k-3)^2 = -1,
\]
then the \qs graph is optimal in $\G(v,e)$ for all $0 \leq e \leq
m$.  But if the lattice point $(v,k_0)$ is between $H$ and
$H_{\text{low}}$, then there exists a value of $e$ in the interval
$4 \leq e \leq m$ such that the \qc graph is optimal and the \qs
graph is not optimal.  Of course, if the lattice point $(v,k_0)$ is
on $H$, then the \qs graph is optimal for all $0 \leq e \leq m$ but
the \qc graph is also optimal for $\binom{k_0}{2} \leq e \leq m$.
Apparently, the density limit
\[
\lim_{v \rightarrow \infty} \frac{n(v)}{v}
\]
from Corollary \ref{cor:density} depends on the density of lattice
points $(v,k)$ in the region between $H_{\text{high}}$ and $H$.

We can give a heuristic argument to suggest that the limit is
$2-\sqrt{2}$.  The asymptotes for the three hyperbolas are
\begin{eqnarray*}
A: v-\frac{5}{2} &=& \sqrt{2}\left(k-\frac{3}{2}\right) \\
A_{\text{low}}: v-\frac{1}{2}&=&\sqrt{2} \left(k-\frac{1}{2} \right)\\
A_{\text{high}}: v-\frac{1}{2} &=& \sqrt{2} \left( k+\frac{1}{2} \right),
\end{eqnarray*}
and intersect the $k$-axis at
\begin{eqnarray*}
k&=&\frac{6-5\sqrt{2}}{4}\\\\
k_{\text{low}}&=&\frac{2-\sqrt{2}}{4}\\
k_{\text{high}}&=&\frac{-2-\sqrt{2}}{4}.
\end{eqnarray*}
The horizontal distance between $A_{\text{high}}$ and $A_{\text{low}}$ is
\[
k_{\text{low}} - k_{\text{high}}=1
\]
and the horizontal distance between $A_{\text{high}}$ and $A$ is
\[
k-k_{\text{high}} = 2-\sqrt{2}.
\]

To make the plausibility argument rigorous, we need a theorem of Weyl \cite[Satz 13, page 334]{We}, \cite[page 92]{Ko}:

For any real number $r$, let $\langle r \rangle$ denote the fractional part of $r$.  That is, $\langle r \rangle$ is the unique number in the half-open interval $[0,1)$ such that $r-\langle r \rangle$ is an integer.
Now let $\beta$ be an irrational real number. Then the sequence $\langle n\beta \rangle$, $n=1,2,3, \ldots$, is uniformly distributed in the interval $[0,1)$.

In our problem, the point  $(v, k_0)$ is between the hyperbolas $H_{\text{low}}$ and $H_{\text{high}}$ and, with few exceptions, $(v, k_0)$ is also between the asymptotes $A_{\text{low}}$ and $A_{\text{high}}$.  To be precise, suppose that $(v,k_0)$ satisfies Inequalities (\ref{eqn:between}).  We need an easy fact from number theory here.  Namely that $y^2-2x^2 \equiv -1 \pmod{8}$ for all odd integers $x,y$.  Thus
\[
 2(2k_0-1)^2 < (2v-1)^2 < 2(2k_0+1)^2,
\]
unless $(2v-1)^2 -2(2k_0-1)^2 =-1$.  (These are the exceptions.)  But for all other points $(v,k_0)$ we have
\[
\sqrt{2} \left( k_0-\frac{1}{2} \right) < v-\frac{1}{2} < \sqrt{2} \left( k_0+\frac{1}{2} \right).
\]
Thus
\[
0<\frac{\sqrt{2}}{2} \left( v-\frac{1}{2} \right) + \frac{1}{2} - k_0 <1
\]
and so
\[
\frac{\sqrt{2}}{2} \left( v-\frac{1}{2} \right) + \frac{1}{2} - k_0=
\left\langle  \frac{\sqrt{2}}{2} \left( v-\frac{1}{2} \right) + \frac{1}{2} \right\rangle.
\]
Next, consider the condition $q_0(v,k_0) \geq 0$, which is equivalent to
\[
(2v-5)^2-2(2k_0-3)^2 \geq -1.
\]
Unless $(2v-5)^2-2(2k_0-3)^2 = -1$,  $q_0(v,k_0) \geq 0$ is equivalent to
\[
\left\langle  \frac{\sqrt{2}}{2} \left( v-\frac{1}{2} \right) + \frac{1}{2} \right\rangle>\sqrt{2}-1.
\]
To summarize, if $(v,k_0)$ does not satisfy either of these Pell's Equations
\[
(2v-1)^2 -2(2k_0-1)^2 =-1, \quad
(2v-5)^2-2(2k_0-3)^2 = -1,
\]
then $q_0(v,k_0) \geq 0$ if and only if
\[
\sqrt{2}-1< \left\langle  \frac{\sqrt{2}}{2} \left( v-\frac{1}{2} \right) + \frac{1}{2} \right\rangle<1.
\]
From Weyl's Theorem, we know that the fractional part in the above inequality is uniformly distributed in the interval $[0,1)$.  Since the density of the values of $v$ for which $(v,k_0)$ is a solution to one of the Pell's Equations above is zero, then $\lim_{v \rightarrow \infty} n(v)/v=1-(\sqrt{2}-1)=2-\sqrt{2}$. The proof of Corollary \ref{cor:density} is complete.

\section{Proofs of Theorems \ref{thm:family5_6}, \ref{thm:family3}, and \ref{thm:family4}}

We first prove Theorem \ref{thm:family5_6}. If $\pi _{1.2}$ and $\pi _{1.3}$ are optimal partitions, then according to Theorem \ref{thm:main2}, $j^{\prime }=3$, $k^{\prime}\geq j^{\prime }+2=5$, and so $v\geq 2k^{\prime }-j^{\prime }\geq7$. In addition the quasi-star partition is optimal, that is, $S(v,e)\geq C(v,e)$. Thus by Corollary \ref{cor:strictuniform}, either $e\geq \binom{v}{2}-3$ or $e\leq m+v/2=\binom{v}{2}/2+v/2$.
If $e\geq \tbinom{v}{2}-3$ and since $j^\prime=3$, then $k^\prime \leq 3$ contradicting $k^\prime \geq 5$. Thus $e\leq \frac{1}{2}\tbinom{v}{2}+\frac{v}{2}$. Since $2k^\prime-3\leq v$ and $e=\tbinom{v}{2}-\tbinom{k^{\prime }+1}{2}+3$, then
\[
3+\frac{1}{2}\binom{v}{2}\leq \binom{k^{\prime
}+1}{2}+\frac{v}{2}\leq \binom{\left( v+3\right)
/2+1}{2}+\frac{v}{2}.
\]
Therefore $7\leq v\leq 13$. In this range of $v$, the only pairs $(v,e)$ that satisfy all the required inequalities are $(v,e)=(7,9)$
or $(9,18)$.

Using the relation between a graph and its complement described below Equation (\ref{eqn:SCconnection}), we conclude that if $\pi _{2.2}$
and $\pi_{2.3}$ are optimal partitions, then $(v,e)=(7,12)$ or $(9,18)$.

As a consequence, we see that the pair $(9,18)$ is the only
candidate to have six different optimal partitions. This is in fact
is the case. The six graphs and partitions are depicted in Figure
\ref{fig:sixoptimal}. We note here that Byer \cite{By} also observed
that the pair $(v,e)=(9,18)$ yields six different optimal graphs.
Another consequence is that the pairs $(7,9)$ and $(7,12)$ are the
only candidates to have five different optimal partitions. For the
pair $(7,9)$, the partitions $\pi _{1.1},\pi _{1.2},\pi _{1.3},\pi
_{2.1}$ and $\pi _{2.2}$ all exist and are optimal. However,
$\pi_{1.3}=\pi _{2.2}$. Thus the pair $(7,9)$ only has four distinct
optimal partitions. Similarly, for the pair $(7,12)$ the partitions
$\pi _{1.1},\pi _{1.2},\pi_{2.1},\pi _{2.2}$ and $\pi _{2.3}$ all
exist and are optimal, but $\pi_{1.2}=\pi _{2.3}$. So there are no
pairs with five optimal partitions, and thus all other pairs have at
most four optimal partitions. Moreover, $S(v,e)=C(v,e)$ is a
necessary condition to have more than two optimal partitions, since
any pair other than $(7,9)$ or $(7,12)$ must satisfy that both
$\pi_{1.1}$ and $\pi _{2.1}$ are optimal. The proof of Theorem
\ref{thm:family5_6} is complete.

\begin{figure}[h]
\begin{center}
\epsfig{file=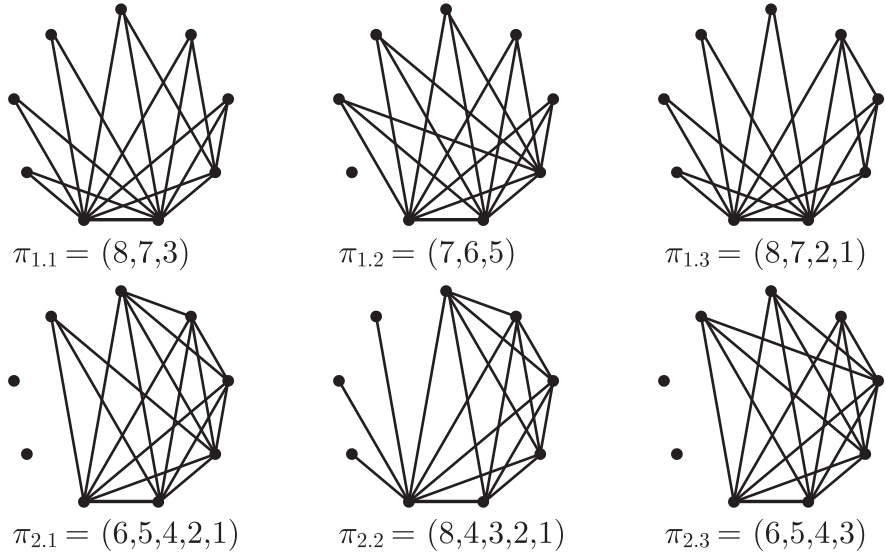,width=4.6in} \caption{$(v,e)=(9,18)$ is
the only pair with six different optimal graphs. For all graphs,
$P_2(\thr(\pi_{i.j}))=\max(v,e)=C(v,e)=S(v,e)=192$}
\label{fig:sixoptimal}
\end{center}
\end{figure}

In Theorem \ref{thm:family3}, $e=\binom{k}{2}=\binom{k+1}{2}-k$ and thus $j=k$. Note that, if $v>5$ and $k$ satisfy Equation (\ref{eqn:family3}), then $k+2<v<2k-1$, and so $k\geq 4$. Thus $e=\binom{v}{2}-\binom{k+2}{2}+(2k+2-v)$ with $4\leq 2k+2-v\leq k+1$, that is, $k^\prime =k+1$ and $j^\prime =2k+2-v$. Hence, $\pi _{1.1}=(v-1,v-2,\ldots ,k+2,2k+2-v)$ and $\pi _{2.1}=(k-1,...,1)$ (which always exist) are different because $2k+2-v\geq 4>1$. The partition $\pi _{1.2}=(v-2,...,k)$ exists because $k\leq v-3$, and it is different to $\pi _{2.1}$ because $k\geq 4>1$ ($\pi_{1.2}\neq \pi _{1.1}$ by definition). Finally, the partitions $\pi_{1.3},\pi _{2.2},$ and $\pi _{2.3}$ do not exist because $j^\prime =2k+2-v\geq 4,k+1>k-1=2k-j-1,$ and $j=k\geq 4$, respectively. Theorem \ref{thm:family3} is proved.

Now, if $v$ and $k$ satisfy Equation (\ref{eqn:family4}), then $\frac{1}{2}\binom{v}{2}=\binom{k+1}{2}-3$. Moreover, since $v>9$, then $k>(v+3)/2$. Hence, in Theorem \ref{thm:family4}, $e=m=\frac{1}{2} \binom{v}{2}=\binom{k+1}{2}-3=\binom{v}{2}-\binom{k+1}{2}+3$, with $k\geq 3$ because $v>1$. That is, $k=k^\prime$ and $j=j^\prime =3$. Thus $\pi _{1.1}=(v-1,v-2,...,k+1,3),\pi _{1.3}=(v-1,v-2,...,k+1,2,1),\pi
_{2.1}=(k-1,k-2,...,4,3)$, and $\pi _{2.3}=(k-1,k-2,...,4,2,1)$ all exist and are different because $k=v$ does not yield a solution to (\ref{eqn:family4}). Also $\pi _{1.2}$ and $\pi _{2.2}$ do not exist because $2k-j-1=2k^\prime -j^\prime -1=2k-4>v-1$. Theorem \ref{thm:family4} is proved.

\section{Pell's Equation} \label{sec:Pell}

Pell's Equation
\begin{equation} \label{eqn:Pell}
V^2-2J^2=P,
\end{equation}
where $P \equiv -1 \pmod{8}$, appears several times in this paper. For example, a condition for the equality  of $S(v,e)$ and $C(v,e)$ in Theorem \ref{thm:main3}
involves the Pell's Equation $(2v-5)^2-2(2k_0-3)^2=-1$.  And in Theorem \ref{thm:family4}, we have $(2v-1)^2-2(2k+1)^2=-49$.  There
are infinitely many solutions to each of these equations. In each instance, $V$ and $J$ in Equation (\ref{eqn:Pell}) are positive
odd integers and $P\equiv -1 \pmod{8}$ .  The following lemma describes the solutions to the fundamental Pell's Equation.
\begin{lemma}\cite{HW} \label{lem:Pell}
All positive integral solutions of
\begin{equation} \label{eqn:Pell-1}
V^2-2J^2=-1
\end{equation}
are given by
\[
V+J\sqrt{2} = (1+\sqrt{2})(3+2\sqrt{2})^n,
\]
where $n$ is a nonnegative integer.

\end{lemma}
It follows from the lemma that if $(V,J)$ is a solution to Equation (\ref{eqn:Pell-1}), then both $V$ and $J$ are odd.  We list the first several
solutions to Equation (\ref{eqn:Pell-1}):
\[
\begin{array}{r|rrrrr}
V&1&7&41&239&1393\\ \hline J&1&5&29&169&985
\end{array}
\]

Now let us consider the equation $(2v-3)^2-2(2k-1)^2=-1$ from
Theorem \ref{thm:family3}.  Since all of the positive solutions
$(V,J)$ consist of odd integers, the pair $(v,k)$ defined by
\[
v=\frac{V+3}{2}, \qquad k=\frac{J+1}{2}
\]
are integers and satisfy Equation (\ref{eqn:family3}).  Thus there
is an infinite family of values for $v>5$ such that there are
exactly 3 optimal partitions in $\ds(v,e)$ where $e=\binom{k}{2}$.
The following is a list of the first three values of $v,k,e$ in this
family:
\[
\begin{array}{r|rrrrr}
v&22&121&698\\ \hline k&15&85&493\\ \hline e&105&3570&121278
\end{array}
\]

Next, consider Equation   (\ref{eqn:family4}) from Theorem
\ref{thm:family4} and the corresponding Pell's Equation:
\[
V^2-2J^2=-49.
\]
A simple argument using the norm function, $N(V+J\sqrt{2})=V^2-2J^2$
shows that all positive integral solutions are given by
\[
\begin{array}{rrrrrr}
V+J\sqrt{2}& = & (1+5\sqrt{2})(3+2\sqrt{2})^n, &(7+7\sqrt{2})(3+2\sqrt{2})^n,& or & (17+13\sqrt{2})(3+2\sqrt{2})^n,
\end{array}
\]
where $n$ is a nonnegative integer.  The first several solutions are
\[
\begin{array}{r|rrrrrrr}
V&1&7&17&23&49&103&137\\ \hline J&5&7&13&17&35&73&97
\end{array}.
\]
Thus the pairs $(v,k)$, defined by
\[
v=\frac{V+1}{2}, \qquad k=\frac{J-1}{2}
\]
satisfy Equation (\ref{eqn:family4}).  The first three members,
$(v,k,e)$  of this  infinite family of partitions $\ds(v,e)$ with
$v>9$, $e=\binom{v}{2}/2$, and exactly 4 optimal partitions are:

\[
\begin{array}{r|rrrrrr}
v&12&25&52&69\\ \hline k&8&17&36&48\\ \hline e&33&150&663&1173
\end{array}
\]

The Pell's Equation
\begin{equation} \label{eqn:Pellmain3}
4q_0(v)=(2v-5)^2-2(2k_0-3)^2+1=0
\end{equation}
appears in Theorem \ref{thm:main3}.  Here again there are infinitely many solutions to the equation $(2v-5)^2-2(2k-3)^2=-1$ starting
with:
\[
\begin{array}{r|rrrrrrrr}
v&2&2&3&3&6&23&122\\ \hline k&1&2&1&2&4&16&86
\end{array}.
\]

The proof of Corollary \ref{cor:uniform} requires infinitely many solutions to the equation $q_0(v)=-2$, which is equivalent to the Pell's Equation
\begin{equation} \label{eqn:PellCor1}
(2v-5)^2-2(2k-3)^2=-9.
\end{equation}

All its positive integral solutions are given by
\[
v=\frac{V+5}{2}, \qquad k=\frac{J+3}{2}, \qquad V+J\sqrt{2}=(3+3\sqrt{2})(3+2\sqrt{2})^n,
\]
where $n$ is a nonnegative integer. The first several solutions are
\[
\begin{array}{r|rrrrrrrr}
v&3&12&63&360&2091\\ \hline k&2&8&44&254&1478
\end{array}
\]
The proof of Corollary \ref{cor:strictuniform} requires infinitely many solutions to the Pell's Equation
\begin{equation} \label{eqn:PellCor2a}
(2v-3)^2-2(2k-1)^2=7,
\end{equation}
and infinitely many solutions to the Pell's Equation
\begin{equation} \label{eqn:PellCor2b}
(2v-3)^2-2(2k-1)^2=-1.
\end{equation}
All positive integral solutions to (\ref{eqn:PellCor2a}) are given by
\[
v=\frac{V+3}{2}, \qquad k=\frac{J+1}{2}, \qquad V+J\sqrt{2}=(3+\sqrt{2})(3+2\sqrt{2})^n, \qquad (5+3\sqrt{2})(3+2\sqrt{2})^n,
\]
where $n$ is a nonnegative integer. The first several solutions are
\[
\begin{array}{r|rrrrrrrr}
v&3&4&8&15&39&80\\ \hline k&1&2&5&10&27&56
\end{array}
\]
We have shown that Equation (\ref{eqn:PellCor2b}) has infinitely many solutions, as it is the same equation that appears in Theorem \ref{thm:family3}. However, in Corollary \ref{cor:strictuniform}, $k$ must be $k_0$, the unique integer that satisfies Inequality (\ref{eqn:k0}). This condition is also necessary for Equations (\ref{eqn:Pellmain3}), (\ref{eqn:PellCor1}), and (\ref{eqn:PellCor2a}). In other words, we must show that for $v$ large enough, every solution $(v,k)$ to one of the Equations (\ref{eqn:Pellmain3}), (\ref{eqn:PellCor1}), or (\ref{eqn:PellCor2a}), satisfies Inequality
(\ref{eqn:k0}). We do this only for Equation (\ref{eqn:Pellmain3}) as all other cases are similar.

\begin{lemma} \label{lem:k)=k}
Let $(v,k)$ be a positive integral solution to Equation
(\ref{eqn:Pellmain3}) with $v>3$.  Then $(v,k)$ satisfies Inequality
(\ref{eqn:k0}).  That is $k=k_0$.
\end{lemma}

\begin{proof}
Suppose that $(v,k)$ is a solution to Equation (\ref{eqn:Pellmain3}) with $v>3$. Then $k<v<2k$.
Inequality (\ref{eqn:k0}) consists of two parts, the first of which
is
\[
\binom{k}{2} \leq \frac{1}{2} \binom{v}{2}.
\]
To prove this part, we compute
\[
\frac{1}{2} \binom{v}{2}-\binom{k}{2}=\frac{1}{2} \binom{v}{2}-\binom{k}{2}-\left((2v-5)^2-2(2k-3)^2+1\right)/16=(v-k)-\frac{1}{2}>0.
\]

The second part of Inequality (\ref{eqn:k0}) is
\[
\frac{1}{2} \binom{v}{2} \leq \binom{k+1}{2}.
\]
This time, we have
\[
\binom{k+1}{2}-\frac{1}{2} \binom{v}{2}=\binom{k+1}{2}-\frac{1}{2}
\binom{v}{2}+\left((2v-5)^2-2(2k-3)^2+1\right)/16=2k-v+\frac{1}{2}>0.
\qedhere
\]
\end{proof}

\textbf{Acknowledgement}. We are grateful to an anonymous referee
who made us aware of Byer's work.

\end{sloppypar}
\end{document}